\documentclass{siamltex}
\usepackage{amsmath,amssymb,textcomp,stmaryrd,xifthen,psfrag,graphicx,color}
\usepackage{color}

\usepackage{rotating}

\SetSymbolFont{stmry}{bold}{U}{stmry}{m}{n} 
\usepackage[T1]{fontenc}
\date{}

\newtheorem{cor}[theorem]{Corollary}
\newtheorem{remark}[theorem]{Remark}

\newcommand{\dual}[2]{\langle#1\hspace*{.5mm},#2\rangle}
\newcommand{\vdual}[2]{(#1\hspace*{.5mm},#2)}
\newcommand{\abs}[1]{\vert #1 \vert}
\newcommand{\norm}[3][]{#1\|#2#1\|_{#3}}
\newcommand{\enorm}[2][]{#1|\hspace*{-.3mm}#1|\hspace*{-.3mm}#1|#2#1|\hspace*{-.3mm}#1|\hspace*{-.3mm}#1|}

\newcommand{\wat}{\widehat}
\def\pwnabla{\nabla_\TT}
\def\ccurl{{\bf curl\,}}
\def\div{{\rm div\,}}
\def\pwdiv{ {\rm div}_{\TT}\,}
\newcommand{\jump}[1]{[#1]}
\newcommand{\hp}{\mathrm{hp}}

\def\pwDelta{\Delta_\TT}

\newcommand{\opt}{{\rm opt}}

\def\eps{\varepsilon}

\newcommand{\R}{\ensuremath{\mathbb{R}}}
\newcommand{\HH}{\ensuremath{\mathbf{H}}}
\newcommand{\cH}{\ensuremath{\mathcal{H}}}
\newcommand{\LL}{\ensuremath{\mathbf{L}}}

\newcommand{\vv}{\ensuremath{\mathbf{v}}}
\newcommand{\TT}{\ensuremath{\mathcal{T}}}
\newcommand{\cS}{\ensuremath{\mathcal{S}}}
\newcommand{\el}{\ensuremath{T}}

\newcommand{\OO}{\ensuremath{\mathcal{O}}}

\newcommand{\GG}{\ensuremath{\mathbf{G}}}

\newcommand{\zz}{\ensuremath{\mathbf{z}}}
\newcommand{\ww}{\ensuremath{\mathbf{w}}}

\newcommand{\U}{{U}}
\newcommand{\V}{{V}}

\newcommand{\ssigma}{{\boldsymbol\sigma}}
\newcommand{\ttau}{{\boldsymbol\tau}}
\newcommand{\llambda}{{\boldsymbol\lambda}}
\newcommand{\qq}{{\boldsymbol{q}}}
\newcommand{\uu}{\boldsymbol{u}}

\newcommand{\n}{\boldsymbol{n}}

\DeclareMathOperator*{\dist}{dist}

\newcommand{\utrace}{{\wat u}}
\newcommand{\ua}{{\wat u^a}}
\newcommand{\ub}{{\wat u^b}}
\newcommand{\ubext}{{u^b}}
\newcommand{\sigmatrace}{{\wat\sigma}}
\newcommand{\sigmaa}{{\wat\sigma^a}}
\newcommand{\sigmab}{{\wat\sigma^b}}

\newcommand{\MM}{\ensuremath{\mathcal{M}}}
\newcommand{\cc}{\ensuremath{\mathbf{c}}}
\title{A robust DPG method for singularly perturbed reaction-diffusion problems
\thanks{Supported by CONICYT through FONDECYT projects 1150056, 3140614,
        Anillo ACT1118 (ANANUM), and by NSF under grant DMS-1318916.}}
\author{
Norbert Heuer\thanks{
Facultad de Matem\'aticas, Pontificia Universidad Cat\'olica de Chile,
Avenida Vicu\~na Mackenna 4860, Macul, Santiago, Chile,
email: {\tt nheuer@mat.uc.cl}}
\and
Michael~Karkulik\thanks{
Departamento de Matem\'atica, Universidad T\'ecnica Federico Santa Mar\'\i a,
Avenida Espa\~na 1680, Valpara\'\i so, Chile,
email: {\tt michael.karkulik@usm.cl}}
}

\begin{document}
\maketitle
\begin{abstract}
We present and analyze a discontinuous Petrov-Galerkin method with optimal test functions
for a reaction-dominated diffusion problem in two and three space dimensions.
We start with an ultra-weak formulation that comprises parameters $\alpha$, $\beta$
to allow for general $\eps$-dependent weightings of three field variables
($\eps$ being the small diffusion parameter).
Specific values of $\alpha$ and $\beta$ imply robustness of the method, that is,
a quasi-optimal error estimate with a constant that is independent of $\eps$.
Moreover, these values lead to a norm for the field variables that is known to be balanced
in $\eps$ for model problems with typical boundary layers.
Several numerical examples underline our theoretical estimates and reveal
stability of approximations even for very small $\eps$.


\end{abstract}
\section{Introduction}

In this paper we analyze the discontinuous Petrov-Galerkin (DPG) method with optimal test functions
for the following singularly perturbed problem of reaction-dominated diffusion,
\begin{subequations}\label{model}
\begin{align}
  -\eps\Delta u + u &= f \quad\text{ in } \Omega, \label{model:1}\\
  u &= 0 \quad\text{ on } \Gamma.
\end{align}
\end{subequations}
Here, $\Omega\subset\R^d$ ($d=2,3$) is a bounded, simply connected Lipschitz polygonal/polyhedral domain
with boundary $\Gamma:=\partial\Omega$. Throughout, we assume that $0<\eps\leq 1$ and
$f$ will be taken from $L_2(\Omega)$.
Such problems appear in applications, e.g., when solving nonlinear reaction-diffusion problems
by the Newton method and in implicit time-discretizations with small time steps of parabolic
reaction-diffusion problems.

The objective of this paper is to push DPG techniques to the limit: For this academic model problem,
\emph{how can we design a DPG method that robustly controls the solution in a norm as strong as possible?}
We will answer this question without using a particular knowledge of the solution
(like the existence of boundary layers) and without using specific meshes.
In this way we hope that our study gives new insight into DPG techniques that will be useful for practical
problems beyond this academic model case.
To be clear, we will not be able to beat approximation properties of a specifically designed method
(square domain and finite elements or finite differences on Shishkin meshes). In contrast, our method
gives robust control in a stronger norm and in general situations.

The numerical approximation of \eqref{model} is notoriously difficult due to the presence of
boundary layers in the solution and due to the deteriorating $H^1(\Omega)$-ellipticity
(of its Dirichlet bilinear form) when $\eps\to 0$.
For an overview of methods for singularly perturbed problems we refer to \cite{RoosST_08_RNM},
and to \cite{Linss_10_LAM} for specific constructions of layer-adapted meshes.
Whereas there is plenty of literature on convection-dominated diffusion problems, the treatment of
reaction-dominated problems is more scarce. Most authors consider very specific domains (like intervals
or squares), specific meshes (e.g., Shishkin meshes) and/or provide an error analysis in
$L_2$- or standard energy norms, cf., e.g., \cite{ApelL_98_AMR,LiN_98_UCF,Lin_08_DDL,XenophontosF_03_UAS}.

For small $\eps$, typical solutions $u_\eps$ of~\eqref{model} contain boundary layers of the type
$\exp(-\alpha\dist(x,\partial\Omega)/\sqrt{\eps})$, cf.~\cite[Thm.~2.3.4]{Melenk_02_hpF}.
Considering the \emph{standard} energy norm (induced by a standard weak form of \eqref{model}),
$\norm{\cdot}{\eps}^2=\norm{\cdot}{L_2(\Omega)}^2 + \eps\norm{\nabla \cdot}{L_2(\Omega)}^2$,
it turns out that both contributions to the norm are not equilibrated:
$\norm{u_\eps}{L_2(\Omega)}^2 = O(1)$ but $\eps\norm{\nabla u_\eps}{L_2(\Omega)}^2 = o(1)$
when $\eps\to 0$. Hence, for small $\eps$, the standard energy norm controls essentially the $L_2(\Omega)$ norm and
not the gradient.

Relatively recently, Lin and Stynes \cite{LinS_12_BFE} proposed to use a balanced norm
(which is stronger than the standard energy norm) where the concentration $u$ (the original unknown),
the flux $\ssigma\sim\nabla u$, and $\rho\sim\Delta u$ are weighted by appropriate powers of the diffusion
coefficient $\eps$ (assuming the reaction coefficient to be fixed or bounded) so that their $L_2$-norms
are of the same magnitude when $\eps\to 0$ (which is not the case for the standard energy norm).
They give a variational formulation and discrete method that
allows for a robust, quasi-optimal error estimate in this balanced norm.
Later, Roos and Schopf \cite{RoosS_15_CSB} provided an analysis for a standard Galerkin method
where they proved a quasi-optimal error estimate in the balanced norm, however, restricted to a square
domain with precise characterization of boundary layers and using Shishkin meshes.
In \cite{MelenkX_16_REC}, Melenk and Xenophontos analyzed the $hp$-version of the finite element method
with appropriate mesh refinement at boundary layers. They also consider the balanced norm and prove robust
exponential convergence, again making use of the specific knowledge of boundary layers and
restricted to one and two space dimensions.

In this paper we systematically develop a DPG scheme with optimal test functions for the approximation
of problem \eqref{model}. Main objective is robustness of the scheme.
This means that an $L_2$-type norm of field variables is controlled by the so-called \emph{energy norm} of
the DPG formulation with a constant that is independent of $\eps$.
The energy norm of the DPG error can be calculated (when using truly optimal test functions)
and thus gives robust a posteriori error control of the field variables.
The DPG method is usually based
on an (ultra-weak) variational formulation of a first-order system and uses a discontinuous Galerkin (DG)
setting. Such a DG approach allows for the efficient generation of optimal test functions
that guarantee discrete stability.
In this combination, the method has been developed by Demkowicz and Gopalakrishnan, see,
e.g., \cite{DemkowiczG_11_ADM,DemkowiczG_11_CDP}.
Recently, DPG technology has been extended by several authors, and in slightly different forms,
to convection-dominated diffusion problems 
\cite{BroersenS_14_RPG,BroersenS_15_PGD,ChanHBTD_14_RDM,DemkowiczH_13_RDM}.
Here, the central idea is to find a DPG setting
(bilinear form, test and ansatz Hilbert spaces with corresponding norms) so that the employment
of optimal test functions (whose definition is automatically given by the setting) produces discrete
inf-sup numbers that are bounded from below by a constant that is independent both of the discrete
ansatz space and perturbation parameters.

Despite of its simpler appearance, reaction-dominated diffusion without convection is harder to
approximate robustly than convection-dominated diffusion. In the latter case, the convective term
allows to control part of the flux ($\ssigma\sim\nabla u$) and this is essential to control
the concentration ($u$ in $L_2$), see \cite{DemkowiczH_13_RDM} for details.
In \cite{DemkowiczH_RDP}, Demkowicz and Harari present attempts to establish a robust DPG setting
for reaction-dominated diffusion. However, similar strategies as in \cite{DemkowiczH_13_RDM}
(but without ultra-weak formulation) lead to conflicts to either weight the ansatz space or the test space
(their norms) with higher powers of $\eps$. Playing with different $\eps$-weightings leads to simple
re-scalings of norms and does not result in a robust setting where optimal test functions can be calculated
with acceptable cost. In this paper, we reconsider the approach to design a robust DPG scheme for
reaction-dominated problems.

Considering reaction-diffusion problems with small diffusion, a natural way to control (parts of) the flux
in a variational setting is to test with second-order derivatives (e.g., the Laplacian) of test functions.
This strategy has been pursued in \cite{LinS_12_BFE} and is also the basis of our variational setting.
When trying to use a ``standard'' ultra-weak formulation where one tests one equation with test functions
as well as with their piecewise Laplacian, it turns out that such a formulation is not well posed: the existence
of a solution is not guaranteed.
To circumvent this problem, we introduce an additional unknown.
Not surprisingly, having tested with the Laplacian of test functions, an appropriate unknown is the Laplacian of the original
unknown. Our DPG approach therefore leads to three field variables: the concentration $u$, the flux
$\ssigma\sim\nabla u$, and $\rho\sim\Delta u$. Without further analysis it is unclear how the weighting
(with respect to the diffusion parameter $\eps$) of the three variables and involved norms should be.
We will therefore introduce two parameters $\alpha,\beta\ge 0$ that serve as powers of $\eps$ in the weighting
of Sobolev norms and several estimates. We initially start with unspecified non-negative parameters.
The quest for robustness will later fix their values at $\beta=2\alpha=1/2$,
cf.~Remark~\ref{rem_alpha_beta} in \S\ref{sec_stab} below. Specifically, our method controls the field variables
robustly in the norms $\|u\|$, $\|\ssigma\|$ with $\ssigma=\eps^{1/4}\nabla u$,
and $\|\rho\|$ with $\rho=\eps^{3/4}\Delta u$
(though we actually define $\rho=\div\ssigma$ and control $\eps^{1/2}\|\rho\|$).
This corresponds to the very norm proposed by Lin and Stynes in \cite{LinS_12_BFE} and which they
show to be balanced when $\eps\to 0$. In particular, on a unit square and with Shishkin
meshes, they prove a uniform (in $\eps$) best approximation error estimate in this balanced norm.

To resume, the requirement of robustness of our DPG ansatz leads to a setting in a norm 
whose components for the field variables are known to be balanced for typical boundary layers.
This analysis applies to two and three dimensions and is independent of the specific domain.
In particular, we do not need any knowledge of possible boundary layers or the solution itself and
we do not use specific meshes.
In this general setting, we are able to prove a robust error estimate for the error in the balanced norm,
bounded quasi-uniformly by the best approximation in the energy norm given by the variational formulation
(not to be confused with the standard energy norm).

There is a small catch we have not resolved so far. Bounding the energy norm by the balanced norm,
e.g., to establish convergence orders, this leads to
a sub-optimality in one of the trace variables whose approximation error is multiplied by $\eps^{-1/4}$.
We do not study best approximation convergence orders in the balanced norm for typical
boundary layers. The best approximation of field variables (in two dimensions on a square) has been
analyzed in \cite{LinS_12_BFE}, and an analysis of the trace variables is an open problem.
For a detailed discussion of the sub-optimality we refer to Remark~\ref{rem_b} in \S\ref{sec_b}.

In practice, optimal test functions have to be approximated. For fixed polynomial degrees and
non-perturbed problems, this \emph{practical} DPG method has been analyzed
by Gopalakrishnan and Qiu \cite{GopalakrishnanQ_14_APD},
with explicit results for the Poisson equation and linear elasticity.
Our analysis considers the \emph{ideal} DPG method with optimal test functions.
We do not analyze the influence of approximating optimal test functions. This is an open problem.
However, in practice, the ``crime'' of approximating optimal test functions of reaction-dominated
diffusion is self correcting through adaptivity. Let us underline this statement with the following
heuristical argument. Error estimation (that is error calculation when exactly
resolving optimal test functions) is an integral part of the DPG method \cite{DemkowiczGN_12_CDP}.
Using this estimation to steer adaptivity, boundary layers of the unknown solution are resolved and
in this way, optimal test functions are also well approximated.
Indeed, the optimal test functions needed to ensure the discrete inf-sup property
(they solve adjoint problems) and the solution of the original problem have boundary layers at
the same locations. This is due to the selfadjointness of the problem.
In comparison, solutions to convection-diffusion problems typically
have layers at the outflow boundary whereas their adjoint problems have boundary layers at the inflow
(of the original problem). Then, adaptivity aiming at the original problem does not produce meshes that
approximate well optimal test functions.
In contrast, adaptivity for reaction-dominated diffusion automatically aims at robustness
(by approximating the test functions increasingly well) and good approximation properties
(boundary layers of unknown functions are detected) at the same time.
Our numerical experiments confirm this interpretation
in the sense that robustness is always obtained and efficiency is achieved when adaptivity
eventually resolves boundary layers.
We also note that our numerical results show robustness of the numerical solutions for extremely
small $\eps$, that is, approximations do not oscillate at boundary layers.

Let us conclude this section with collecting the main results of this paper and remaining open problems.
\begin{itemize}
\item Our search for a variational formulation whose energy norm can robustly control field variables
  ($u$ and $\nabla u$) in $L_2$ led us to a three-field scheme (also containing $\Delta u$ as an unknown).
  We stress that it might be possible to obtain a two-field scheme (having unknowns $u$ and $\nabla u$)
  by imposing more regularity on the test space.
\item Control of the field variables by the (DPG) energy norm is proved by an abstract stability analysis
      (of the adjoint problem).
      The condition of \emph{robust} control leads to the balanced norm of the field variables.
      This is a general outcome in two and three space dimensions without assuming the presence of boundary
      layers or the use of specific meshes.
\item By design of the DPG method with optimal test functions, the error in the energy norm can be calculated
      elementwise. We thus have a robust a posteriori error control of the field variables in the balanced
      norm.
\item In practice, optimal test functions have to be approximated. An analysis of the effect of this approximation
      for singularly perturbed problems is ongoing research.
\item We prove a best approximation property of our DPG scheme in a norm that is balanced in the field
      variables and contains scaled trace norms of skeleton variables. Approximation results
      for this very norm have yet to be produced.
\end{itemize}

The remainder of the paper is as follows. In Section~\ref{sec_DPG} we develop and analyze our DPG method.
Sobolev spaces are introduced in Subsection~\ref{sec_Sobolev}, and in Subsection~\ref{sec_VF} we present
an ultra-weak formulation, the discrete scheme and the main result (Theorem~\ref{thm}). This theorem
provides a robust estimate for the balanced norm by the energy norm. We also give an upper bound of
the energy norm in terms of the balanced norm. In the subsequent subsections we analyze the bilinear form
of the variational formulation and show stability of solutions to the adjoint problem. This stability
implies robustness of the error estimate from Theorem~\ref{thm}. Finally, in Section~\ref{sec_num}, we
present several numerical results.

Throughout the paper, $a\lesssim b$ means that $a\le cb$ with a generic constant $c>0$ that is independent of
involved parameters, functions and the underlying mesh. Similarly, we use the notation $a\simeq b$.

\section{Presentation and analysis of the DPG method} \label{sec_DPG}
\subsection{Introduction to the DPG method}\label{sec_intro_dpg}
We briefly recall the framework and results of the DPG method with optimal test
functions, cf.~\cite{DemkowiczG_11_ADM}. Given a Banach space
$U$, a Hilbert space $V$, and a bilinear form $b:U\times V\rightarrow \R$, we
consider the following three conditions:
\begin{subequations}\label{eq:bb}
  \begin{align}\label{eq:bb:1}
    b(\uu,\vv)= 0 \text{ for all } \vv\in V \implies \uu = 0;\\
    \label{eq:bb:2}
    C_{\rm infsup}\norm{\vv}{V} \leq \sup_{\uu\in U}
    \frac{b(\uu,\vv)}{\norm{\uu}{U}}\quad \text{ for all } \vv\in V;\\
    \label{eq:bb:3}
    b(\uu,\vv)\leq C_{\rm b}\norm{\uu}{U}\norm{\vv}{V}
    \qquad\text{ for all } \uu\in U, \vv\in V.
  \end{align}
\end{subequations}
Here, $C_{\rm infsup}$ and $C_{\rm b}$ are positive constants.
Define the so-called trial-to-test operator $\Theta:U\rightarrow V$ by
\begin{align}\label{eq:ttot}
  \dual{\Theta \uu}{\vv}_V = b(\uu,\vv)\quad \text{ for all }\vv\in V,
\end{align}
or, equivalently, by
\begin{align}
  \Theta = J^{-1}B,
  \label{eq:ttot2}
\end{align}
where $B:U\rightarrow V'$ is the operator corresponding to the bilinear form $b$ and $J:V\rightarrow V'$
is the Riesz operator.
The following result is central to the DPG method and is, in the end, consequence of the Babu\v{s}ka-Brezzi
theory and related references given in the introduction.
\begin{theorem}\label{thm:dpg}
  Suppose that~\eqref{eq:bb:1}--\eqref{eq:bb:3} hold
  for a Banach space $U$, a Hilbert space $V$,
  and a bilinear form $b:U\times V\rightarrow \R$.
  Then, an equivalent norm on $U$ is given by
  \begin{align}\label{norm_E}
    \norm{\uu}{E} := \sup_{\vv\in V}\frac{b(\uu,\vv)}{\norm{\vv}{V}},
    \quad\text{ with }\quad C_{\rm infsup}\norm{\uu}{U} \leq \norm{\uu}{E}
    \le C_{\rm b}\norm{\uu}{U}\quad \forall \uu\in U.
  \end{align}
  Furthermore, for any $L\in V'$, the problem
  \begin{align}\label{thm:dpg:eq:weak}
    \text{ find } \uu\in U \text{ such that }\quad
    b(\uu,\vv) = L(\vv)\quad \text{ for all }\vv\in V
  \end{align}
  has a unique solution, and
  \begin{align}\label{thm:dpg:eq:stab}
    \norm{\uu}{E} = \norm{L}{V'}.
  \end{align}
  In addition, if $U_\hp\subset U$ is a finite-dimensional subspace, then the
  problem
  \begin{align}\label{thm:dpg:eq:discrete}
    \text{ find } \uu_\hp\in U_\hp \text{ such that }\quad
    b(\uu_\hp,\vv_\hp) = L(\vv_\hp)\quad \text{ for all }\vv_\hp\in \Theta(U_\hp)
  \end{align}
  has a unique solution, and
  \begin{align}\label{thm:dpg:eq:approx}
    \norm{\uu-\uu_\hp}{E} = \inf_{\uu_\hp'\in
    U_\hp}\norm{\uu-\uu_\hp'}{E}.
  \end{align}
\end{theorem}
Integral part of the computation of the numerical approximation $\uu_\hp$ in~\eqref{thm:dpg:eq:discrete}
is the generation of $\Theta(U_\hp)$ defined via~\eqref{eq:ttot}.
Unless this can be done analytically,~\eqref{eq:ttot} has to be discretized. This gives the so-called~\emph{practical DPG method},
cf.~\cite{GopalakrishnanQ_14_APD}.
For an efficient discretization, the space $V$ 
and its norm $\norm{\cdot}{V}$ and associated inner product $\dual{\cdot}{\cdot}_V$
have to be chosen in a broken form, i.e., local with respect to elements of some mesh $\TT$:
$V=\prod_{\el\in\TT} V(\el)$, and $\dual{\cdot}{\cdot}_V = \sum_{\el\in\TT}\dual{\cdot}{\cdot}_{V(\el)}$.
Then, a discretization of~\eqref{eq:ttot} amounts to a block-diagonal matrix, with blocks associated to
elements $\el$.
The supremum on the right-hand side of the inf-sup condition~\eqref{eq:bb:2} is usually called the \textit{optimal test norm}
\begin{align}
  \norm{\vv}{\V,\opt} := \sup_{\uu\in U}
    \frac{b(\uu,\vv)}{\norm{\uu}{U}}\quad \text{ for all } \vv\in V,
  \label{norm_V}
\end{align}
and the inf-sup condition~\eqref{eq:bb:2} then renders like
\begin{align}\label{normeq}
  \norm{\vv}{V} \lesssim \norm{\vv}{V,\opt}.
\end{align}
In view of the definition~\eqref{norm_V}, the bound~\eqref{normeq} amounts to the stability of the adjoint problem,
which is hence a major part in DPG analysis.
A feature of the DPG method is that it provides local error indicators. From~\eqref{eq:ttot} and~\eqref{norm_E}
it follows
\begin{align}
    C_{\rm infsup}\norm{\uu-\uu_\hp}{U} \le
  \norm{\uu-\uu_\hp}{E} = \norm{\Theta(\uu-\uu_\hp)}{V} = \norm{J^{-1}(L-B\uu_\hp)}{V}.
  \label{eq:est}
\end{align}
There are two implications of these relations. First, the residual in the $V'$-norm is the
error in energy norm and thus, controls it robustly. Second, if $C_{\rm infsup}$
is a constant independent of possible perturbation parameters then the residual is a
reliable and efficient estimator for the error in the $U$-norm. In this paper, we control
the error of the field variables in the balanced norm (see Theorem~\ref{thm} below),
but not the skeleton variables.

If the norm in $V$ is broken with respect to the mesh $\TT$, then the right-hand side
in \eqref{eq:est} provides a local a posteriori error estimate.
It should be mentioned that convergence of adaptive algorithms based on this estimator
has not been analyzed even for the Poisson problem.
\subsection{Sobolev spaces} \label{sec_Sobolev}
Let us first introduce some notation. For a set $\omega\subset\R^d$, $L_2(\omega)$,
$H^1(\omega)$, $H^1_0(\omega)$, $\HH(\div,\omega)$, and $\HH(\ccurl,\omega)$
are the standard Sobolev spaces with usual norms.
The norm in $L_2(\Omega)$ will be denoted by $\norm{\cdot}{}$.
The dual space of $H^1_0(\Omega)$ is denoted by $H^{-1}(\Omega)$ with norm $\norm{\cdot}{-1}$.
Throughout, spaces with bold face symbols, e.g. $\HH(\div,\omega)$, refer to spaces of vector-valued
functions. The $L_2(\Omega)$ and $\LL_2(\Omega)$ inner products and their extensions by duality
are denoted by $\vdual{\cdot}{\cdot}$. For $\omega\subset\R^d$, $\dual{\cdot}{\cdot}_{\partial\omega}$ refers
to this duality on the boundary of $\omega$. We will also need the spaces
\begin{align*}
   H^1(\Delta,\Omega) &:= \{w\in H^1(\Omega);\; \Delta w\in L_2(\Omega)\},\\
   H^1_0(\Delta,\Omega) &:= H^1(\Delta,\Omega) \cap H^1_0(\Omega).
\end{align*}
The setting of our continuous and discrete formulations is based on broken spaces, related to partitions
$\TT$ of $\Omega$. Let $\TT$ denote such a partition (or mesh) that is compatible with the geometry, i.e.,
$\TT$ is a finite set, and the elements $\el\in\TT$ are mutually disjoint, open sets with
$\bigcup_{\el\in\TT} \overline\el = \overline\Omega$.
Related to $\TT$, we introduce broken Sobolev spaces
\begin{align*}
   \HH(\div,\TT) &:= \{\qq\in \LL_2(\Omega);\; \qq|_\el\in \HH(\div,\el)\ \forall\el\in\TT\},\\
   H^1(\TT) &:= \{w\in L_2(\Omega);\; w|_\el\in H^1(\el)\ \forall\el\in\TT\},\\
   H^1(\Delta,\TT) &:= \{w\in H^1(\TT);\; \Delta w|_\el\in L_2(\el)\ \forall\el\in\TT\},
\end{align*}
and corresponding broken operators $\pwdiv$, $\pwnabla$, $\pwDelta$ which are defined
piecewise with respect to elements.
We also define trace spaces on the skeleton $\cS$ of the mesh. It is convenient to
consider $\cS$ as the collection of boundaries of elements, $\cS:=\{\partial\el;\; \el\in\TT\}$,
rather than a single geometric object, and to define spaces on $\cS$ as product spaces
(of components which are not independent).
Correspondingly, the ``normal vector'' $\n$ on $\cS$ consists of components $\n_\el$ which
are the exterior normal vectors of unit length on $\partial\el$ ($\el\in\TT$).

The space $H^{1/2}_{00}(\cS)$ consists of elements $\utrace=\Pi_{\el\in\TT}\utrace_\el$ 
whose components are traces of $H^1_0(\Omega)$ functions. It is equipped with the norm
\begin{align} \label{norm_ua}
  \norm{\utrace}{1/2,\cS}
  &:=
  \inf \Big\{
    \bigl(\norm{w}{}^2 + \eps^{2\alpha}\norm{\nabla w}{}^2\bigr)^{1/2};\;
    w\in H^1_0(\Omega),\ w|_{\partial\el} = \utrace_\el\ \forall\el\in\TT
  \Big\}.
\end{align}
The space $H^{-1/2}(\cS)$ consists of elements $\sigmatrace=\Pi_{\el\in\TT}\sigmatrace_\el$
whose components are normal components of $\HH(\div,\Omega)$ functions, equipped with the norm
\begin{align} \label{norm_sigmaa}
  \norm{\sigmatrace}{-1/2,\cS}
  &:=
  \inf \Big\{
    \bigl(\norm{\qq}{}^2 + \eps^{2\beta}\norm{\div\qq}{}^2\bigr)^{1/2};\;
    \qq\in \HH(\div,\Omega),\ \qq\cdot\n_\el|_{\partial\el} = \sigmatrace_\el\ \forall\el\in\TT
  \Big\}.
\end{align}
For $v\in H^1(\TT)$ and $\ttau\in\HH(\div,\TT)$,
corresponding dual norms of their jumps are denoted by
\begin{align*}
   &\norm{\jump{v}}{1/2,\cS'}
   :=
   \sup_{\varphi\in H^{-1/2}(\cS)} \frac {\dual{v}{\varphi}}{\norm{\varphi}{-1/2,\cS}}
   \quad\text{with}\quad
   \dual{v}{\varphi}:= \sum_{\el\in\TT} \dual{v}{\varphi_\el}_{\partial\el},
   \\
   &\norm{\jump{\ttau\cdot\n}}{-1/2,\cS'}
   :=
   \sup_{\varphi\in H^{1/2}_{00}(\cS)}
   \frac {\dual{\ttau\cdot\n}{\varphi}}{\norm{\varphi}{1/2,\cS}}
   \quad\text{with}\quad
   \dual{\ttau\cdot\n}{\varphi}
   := \sum_{\el\in\TT} \dual{\ttau\cdot\n_\el}{\varphi_\el}_{\partial\el}.
\end{align*}
Here and in the following, suprema are taken over non-zero elements of spaces.

\subsection{Variational formulation, DPG scheme and main result} \label{sec_VF}
Let us develop a variational formulation of our model problem.
We write~\eqref{model:1} as the first-order system
\begin{equation}\label{fos}
  \eps^{-\alpha}\ssigma - \nabla u = 0,\quad
  \rho - \div\ssigma = 0,\quad
  -\eps^{1-\alpha}\rho + u = f,
\end{equation}
and define
\[
   \U
   :=
   L_2(\Omega)\times\LL_2(\Omega) \times L_2(\Omega)\times H^{1/2}_{00}(\cS) \times
   H^{1/2}_{00}(\cS)\times H^{-1/2}(\cS)\times H^{-1/2}(\cS).
\]
Now, let $\TT$ be a mesh (as defined previously) and let $\ttau$, $\mu$, and $v$ be $\TT$-piecewise smooth functions.
We multiply the first, second and third relations in \eqref{fos}, respectively, by
$\ttau$, $\mu$, and $v-\eps^\beta\pwDelta v$, and integrate piecewise by parts. Then,
 we obtain the following variational formulation.
\emph{Find $(u,\ssigma,\rho,\ua,\ub,\sigmaa,\sigmab)\in \U$ such that
\begin{subequations}\label{b}
\begin{align}
  \eps^{-\alpha} \vdual{\ssigma}{\ttau} + \vdual{u}{\pwdiv\ttau} - \dual{\ua}{\ttau\cdot\n} &= 0\label{b:1}\\
  \vdual{\rho}{\mu} + \vdual{\ssigma}{\pwnabla\mu} - \dual{\sigmaa}{\mu} &= 0\label{b:2}\\
  \left.
  \begin{aligned}
  \eps^{1-\alpha} \vdual{\ssigma}{\pwnabla v} - \eps^{1-\alpha}\dual{\sigmab}{v}
  +\vdual{u}{v}\\
  + \eps^{1-\alpha+\beta}\vdual{\rho}{\pwDelta v} + \eps^{\beta-\alpha} \vdual{\ssigma}{\pwnabla v}
  - \eps^{\beta} \dual{\ub}{\pwnabla v\cdot\n}
  \end{aligned}
  \right\} &= \vdual{f}{v - \eps^\beta\pwDelta v}\label{b:3}
\end{align}
\end{subequations}
for all $(\ttau,\mu,v)\in \V := \HH(\div,\TT)\times H^1(\TT)\times H^1(\Delta,\TT)$.}
The left-hand side of \eqref{b} defines our bilinear form
\begin{align*}
   b(\uu,\vv)
   :=\;
     &\vdual{u}{\pwdiv\ttau+v}
   + \vdual{\ssigma}
           {\eps^{-\alpha}\ttau+\pwnabla\mu+(\eps^{1-\alpha}+\eps^{\beta-\alpha})\pwnabla v}
   + \vdual{\rho}{\mu+\eps^{1-\alpha+\beta}\pwDelta v}\\
   &- \dual{\ua}{\ttau\cdot\n} - \eps^\beta\dual{\ub}{\pwnabla v\cdot\n}
   - \dual{\sigmaa}{\mu} - \dual{\sigmab}{\eps^{1-\alpha}v}
\end{align*}
for $\uu=(u,\ssigma,\rho,\ua,\ub,\sigmaa,\sigmab)$ and $\vv=(\ttau,\mu,v)$.
The right-hand side of \eqref{b} is abbreviated by
$L(\vv):=\vdual{f}{v - \eps^\beta\pwDelta v}$ with $\vv$ as before.

In $\U$ and $\V$ we introduce, respectively, the norms
\begin{align}
  \label{norm_U}
  \norm{(u,\ssigma,\rho,\ua,\ub,\sigmaa,\sigmab)}{\U}^2
  &:=
  \norm{u}{}^2 + \norm{\ssigma}{}^2 + \eps^{2\beta}\norm{\rho}{}^2
  \nonumber \\
  &\quad
  + \norm{\ua}{1/2,\cS}^2 + \norm{\ub}{1/2,\cS}^2
  + \norm{\sigmaa}{-1/2,\cS}^2 + \norm{\sigmab}{-1/2,\cS}^2,
  \\ \nonumber
  \norm{(\ttau,\mu,v)}{\V}^2
  &:=
  \eps^{-2\alpha}\norm{\ttau}{}^2 + \norm{\pwdiv\ttau}{}^2
  + \eps^{-2\beta}\norm{\mu}{}^2 + \norm{\pwnabla\mu}{}^2
  \\ \nonumber
  &\quad
  + \norm{v}{}^2 + \eps^{2(\beta-\alpha)}\norm{\pwnabla v}{}^2
  + \eps^{2(1-\alpha)}\norm{\pwDelta v}{}^2.
\end{align}
The norm in $\V$ is induced by the inner product $\langle \cdot,\cdot\rangle_V$.
Note that in accordance with Section~\ref{sec_intro_dpg}, the space $\V$ and its inner product
are of the broken type.
The optimal test norm \eqref{norm_V} in $\V$ is
\begin{align} \label{norm_Vopt}
   &\norm{(\ttau,\mu,v)}{\V,\opt}^2
   \nonumber\\
   &=
   \norm{\pwdiv\ttau+v}{}^2
   +
   \norm{\eps^{-\alpha}\ttau+\pwnabla\mu+(\eps^{1-\alpha}+\eps^{\beta-\alpha})\pwnabla v}{}^2
   \nonumber
   +
   \eps^{-2\beta}\norm{\mu+\eps^{1-\alpha+\beta}\pwDelta v}{}^2
   \nonumber\\
   &\quad+
   \norm{\jump{\ttau\cdot\n}}{-1/2,\cS'}^2
   +
   \norm{\jump{\eps^\beta\pwnabla v\cdot\n}}{-1/2,\cS'}^2
   +
   \norm{\jump{\mu}}{1/2,\cS'}^2
   +
   \norm{\jump{\eps^{1-\alpha} v}}{1/2,\cS'}^2.
\end{align}
\begin{remark}
We note that the optimal test norm is not of broken type due to the appearance of norms of jumps of test functions.
The related inner product is therefore not appropriate for the calculation of optimal test functions,
cf.~\eqref{eq:ttot}. The corresponding problems would not be local. However, there are advocates of using
the so-called quasi-optimal test norm which consists in replacing the jump terms by (scaled) $L_2$-norms of
test functions. In our case it (its squared value) would be
\begin{align*}
   &\norm{\pwdiv\ttau+v}{}^2
   +
   \norm{\eps^{-\alpha}\ttau+\pwnabla\mu+(\eps^{1-\alpha}+\eps^{\beta-\alpha})\pwnabla v}{}^2
   +
   \eps^{-2\beta}\norm{\mu+\eps^{1-\alpha+\beta}\pwDelta v}{}^2
   \\
   &\qquad+
   c_1\norm{\ttau}{}^2
   +
   c_2\norm{\mu}{}^2
   +
   c_3\norm{v}{}^2
\end{align*}
with appropriate numbers $c_1,c_2,c_2\ge 0$.
Using this norm to define test functions would simplify the stability analysis required for \eqref{normeq}
in the sense that it suffices to bound the $L_2$-norm of test functions.
This is a path analyzed in \cite{NiemiCC_13_ASD} for convection-dominated diffusion, see also \cite{BroersenS_14_RPG}.

In any case, aiming at robustness for singularly perturbed problems,
the calculation of optimal test functions with respect to any appropriate inner product will lead to
singularly perturbed problems. The advantage of using broken spaces is that these are local problems on
elements. In the case of optimal or quasi-optimal test norms, test functions are coupled
so that local problems are more complicated and harder to solve. Additionally, in these cases the singularly
perturbed problems are not of standard type so that it is not straightforward to design and analyze efficient
approximation schemes for optimal test functions. In \cite{NiemiCC_13_ASD}, Niemi, Collier and Calo deal
with this very problem in the case of convection-dominated problems.

In our case with three test functions,
it is non-trivial to analyze and solve the coupled singularly-perturbed problems stemming from
a quasi-optimal test norm. Instead, we prefer to simplify these problems by separating functions.
Calculating test functions with respect to our test norm $\norm{\cdot}{V}$ leads to solving (on elements)
three separate problems with bilinear forms
\begin{align*}
   &\eps^{-2\alpha} (\cdot,\cdot) + (\div\cdot,\div\cdot) \quad\text{for } \ttau, \qquad
   \eps^{-2\beta} (\cdot,\cdot) + (\nabla\cdot,\nabla\cdot) \quad\text{for } \mu,\\
   &\text{and}\quad
   (\cdot,\cdot) + \eps^{2(\beta-\alpha)} (\nabla\cdot,\nabla\cdot)
                 + \eps^{2(1-\alpha)}(\Delta\cdot,\Delta\cdot) \quad\text{for } v,
\end{align*}
with $(\cdot,\cdot)$ denoting the $L_2$-bilinear form on an individual element. Though singularly perturbed,
these are standard elliptic problems so that an analysis of the influence of approximating optimal test functions
appears more accessible. As previously mentioned, our analysis is based on using exact optimal test functions.
\end{remark}

Our main result is the following norm equivalence in $\U$. It induces corresponding error
estimates for the DPG method, recalled by Corollary~\ref{cor_thm}.

\begin{theorem} \label{thm}
Choose $\alpha=1/4$ and $\beta=1/2$. For the setting introduced in this section,
\eqref{eq:bb:1}--\eqref{eq:bb:3} hold with numbers $C_{\rm infsup}$ and $C_{\rm b}$
that depend on $\eps$. More specifically, we have robust control of the field variables in the sense that
\begin{align*}
\lefteqn{
  \norm{u}{} + \norm{\ssigma}{} + \eps^{1/2} \norm{\rho}{}
}
  \\
  &\qquad+ \eps^{3/4} \norm{\ua}{1/2,\cS} + \eps^{1/2} \norm{\ub}{1/2,\cS}
  + \eps^{3/4} \norm{\sigmaa}{-1/2,\cS} + \eps^{5/4} \norm{\sigmab}{-1/2,\cS}
  \lesssim
  \norm{\uu}{E}
\end{align*}
and
\begin{align} \label{upper}
  \norm{\uu}{E}
  &\lesssim
  \norm{u}{} + \norm{\ssigma}{} + \eps^{1/2} \norm{\rho}{}
  + \norm{\ua}{1/2,\cS}
  \nonumber \\
  &\quad + \eps^{-1/4} \norm{\ub}{1/2,\cS}
  + \norm{\sigmaa}{-1/2,\cS} + \eps^{1/4} \norm{\sigmab}{-1/2,\cS}
\end{align}
for any $\uu=(u,\ssigma,\rho,\ua,\ub,\sigmaa,\sigmab)\in \U$.
The constants appearing in both estimates are independent of $\TT$ and $\eps>0$.
\end{theorem}

\begin{proof}
  Technical details of the proof are given in the remainder of this paper. More precisely,
  condition~\eqref{eq:bb:1} is shown in Lemma~\ref{lem:inj}. The inf-sup condition~\eqref{eq:bb:2}
  or, equivalently,~\eqref{normeq} is shown in Corollary~\ref{cor} (the right-hand side involves different scalings of $\eps$ for
  the skeleton terms). The condition~\eqref{eq:bb:3} is shown in Lemma~\ref{lem:b:stab}. Hence, the first bound follows from
  Theorem~\ref{thm:dpg}.
    The second bound follows directly from Lemma~\ref{lem:b:stab}.
\end{proof}

\begin{cor} \label{cor_thm}
Select $\alpha=1/4$, $\beta=1/2$. 
Then, there exist solutions
$\uu=(u,\ssigma,\rho,\ua,\ub,\sigmaa,\sigmab)\in U$ and
$\uu_\hp=(u_\hp,\ssigma_\hp,\rho_\hp,\ua_\hp,\ub_\hp,\sigmaa_\hp,\sigmab_\hp)\in U_\hp$
of~\eqref{thm:dpg:eq:weak} and~\eqref{thm:dpg:eq:discrete}, respectively.
We have the robust error estimate
\begin{align*}
\lefteqn{
  \norm{u-u_\hp}{} + \norm{\ssigma-\ssigma_\hp}{} + \eps^{1/2} \norm{\rho-\rho_\hp}{}
  + \eps^{3/4} \norm{\ua-\ua_\hp}{1/2,\cS} + \eps^{1/2} \norm{\ub-\ub_\hp}{1/2,\cS}
}
  \\
  &+ \eps^{3/4} \norm{\sigmaa-\sigmaa_\hp}{-1/2,\cS} + \eps^{5/4} \norm{\sigmab-\sigmab_\hp}{-1/2,\cS}
  \lesssim
  \inf \{\norm{\uu-\ww}{E};\; \ww\in U_\hp\}.
\end{align*}
The hidden constant is independent of $\TT$, $U_\hp$ and $\eps>0$.
The best approximation in the energy norm can be bounded from above, as in \eqref{upper}.
\end{cor}

\begin{proof}
  In Theorem~\ref{thm} we showed that our setting fulfills the assumptions from Theorem~\ref{thm:dpg}.
  This shows that the continuous and discrete solutions exist uniquely.
  Additionally, the method delivers the best approximation
  in the energy norm, cf.~\eqref{thm:dpg:eq:approx}. Therefore, the norm estimates from Theorem~\ref{thm} prove the statements.
\end{proof}

\subsection{Boundedness and definiteness of the bilinear form} \label{sec_b}\quad\\

\begin{lemma}\label{lem:b:stab}
  For $\alpha, \beta\in[0,1]$ with $\alpha+\beta\le 1$ and $\epsilon>0$,
  the bilinear form $b:\U\times \V\rightarrow\R$ is bounded:
\begin{align*}
   b(\uu,\vv)
   &\lesssim
   \Bigl(  \norm{u}{} + \norm{\ssigma}{} + \eps^\beta \norm{\rho}{}
   \\ & \qquad 
         + \norm{\ua}{1/2,\cS} + \eps^{\alpha+\beta-1} \norm{\ub}{1/2,\cS}
         + \norm{\sigmaa}{-1/2,\cS} + \eps^{1-\alpha-\beta} \norm{\sigmab}{-1/2,\cS}
   \Bigr)
   \norm{\vv}{\V}
   \\
   &\lesssim
   \eps^{\alpha+\beta-1} \norm{\uu}{\U} \norm{\vv}{\V}
   \quad\forall \uu=(u,\ssigma,\rho,\ua,\ub,\sigmaa,\sigmab)\in \U,\ \forall \vv=(\ttau,\mu,v)\in \V.
\end{align*}
\end{lemma}
\begin{proof}
  The volume terms are estimated with the Cauchy-Schwarz inequality.
  The terms on the skeleton are additionally integrated piecewise by parts. More precisely,
  for $w\in H^1_0(\Omega)$ with $\utrace_\el = w|_{\partial\el}$ for all $\el\in\TT$ we obtain
  \begin{align*}
    \dual{\utrace}{\ttau\cdot\n} \leq
    \left( \eps^{-2\alpha}\norm{\ttau}{}^2 + \norm{\pwdiv\ttau}{}^2 \right)^{1/2}
    \left( \norm{w}{}^2 + \eps^{2\alpha}\norm{\nabla w}{}^2 \right)^{1/2}
  \end{align*}
  and
  \begin{align*}
    \eps^\beta \dual{\utrace}{\pwnabla v\cdot\n} &\leq
    \left( \eps^{2(1-\alpha)}\norm{\pwDelta v}{}^2
    + \eps^{2(\beta-\alpha)} \norm{\pwnabla v}{}^2 \right)^{1/2}
    \left( \eps^{2(\alpha+\beta-1)}\norm{w}{}^2
    + \eps^{2\alpha} \norm{\nabla w}{}^2 \right)^{1/2},
  \end{align*}
  that is,
  \[
    \dual{\ua}{\ttau\cdot\n} + \eps^\beta \dual{\ub}{\pwnabla v\cdot\n}
    \lesssim
    \Bigl(\norm{\ua}{1/2,\cS} +  \eps^{\alpha+\beta-1} \norm{\ub}{1/2,\cS}\Bigr)
    \norm{\vv}{\V},\quad \vv=(\ttau,0,v).
  \]
  Furthermore, for $\qq\in\HH(\div,\Omega)$ with $\qq\cdot\n|_{\partial\el}=\sigmatrace_\el$
  for all $\el\in\TT$ it holds that
  \[
     \dual{\sigmatrace}{\mu}
     \le
     \Bigl(\norm{\qq}{}^2 + \eps^{2\beta} \norm{\div\qq}{}^2\Bigr)^{1/2}
     \Bigl(\eps^{-2\beta} \norm{\mu}{}^2 + \norm{\pwnabla\mu}{}^2\Bigr)^{1/2}
  \]
  and
  \begin{align*}
     \eps^{1-\alpha} \dual{\sigmatrace}{v}
     &\le
     \Bigl(  \eps^{2(1-\beta)} \norm{\qq}{}^2
           + \eps^{2(1-\alpha)} \norm{\div\qq}{}
     \Bigr)^{1/2}
     \Bigl(\norm{v}{}^2 + \eps^{2(\beta-\alpha)} \norm{\pwnabla v}{}^2\Bigr)^{1/2},
  \end{align*}
  that is,
  \[
     \dual{\sigmaa}{\mu} + \eps^{1-\alpha} \dual{\sigmab}{v}
     \lesssim
     \Bigl(\norm{\sigmaa}{-1/2,\cS} + \eps^{1-\alpha-\beta} \norm{\sigmab}{-1/2,\cS}\Bigr)
     \norm{\vv}{\V},
     \quad \vv=(0,\mu,v).
  \]
  This concludes the proof of the lemma.
\end{proof}

\begin{remark} \label{rem_b}
The previous lemma establishes uniform boundedness of the bilinear form $b(\cdot,\cdot)$
only if $\alpha+\beta=1$. Unfortunately, our quest for robustness of the DPG scheme will
lead to $\alpha+\beta=3/4$, i.e., $b(\cdot,\cdot)$ will not be bounded uniformly in $\eps$.
By definition of the energy norm (cf.~\eqref{norm_E}) this means that we will not have
a uniform bound $\norm{\cdot}{E}\lesssim\norm{\cdot}{\U}$ in $\U$, but rather the estimate
\eqref{upper} stated in Theorem~\ref{thm}.

It is straightforward to ensure uniform boundedness of $b(\cdot,\cdot)$. Having different
pairs of trace and flux variables, $(\ua,\sigmaa)$ and $(\ub,\sigmab)$, we can use different
norms. More precisely, employing the previously defined norms for $(\ua,\sigmaa)$,
cf.~\eqref{norm_ua} and \eqref{norm_sigmaa}, and
the norms
\begin{align*}
  \norm{\ub}{1/2,\cS,b}
  &:=
  \inf \Big\{
    (\eps^{2(\alpha+\beta-1)}\norm{w}{}^2 + \eps^{2\alpha}\norm{\nabla w}{}^2)^{1/2};\;
    w\in H^1_0(\Omega),\ w|_{\partial\el} = \ub_\el\ \forall\el\in\TT
  \Big\},
  \\
  \norm{\sigmab}{-1/2,\cS,b}
  &:=
  \inf \Big\{
    (\eps^{2(1-\beta)}\norm{\qq}{}^2 + \eps^{2(1-\alpha)}\norm{\div\qq}{}^2)^{1/2};\;
    \qq\in \HH(\div,\Omega),\ \qq\cdot\n_\el|_{\partial\el} = \sigmab_\el\ \forall\el\in\TT
  \Big\},
\end{align*}
for the second pair $(\ub,\sigmab)$, it is easy to show that then the bilinear form
is uniformly bounded. Note that for $\alpha+\beta=3/4$,
$\norm{\cdot}{1/2,\cS,b}$ is stronger than $\norm{\cdot}{1/2,\cS}$
and $\norm{\cdot}{-1/2,\cS,b}$ is weaker than $\norm{\cdot}{-1/2,\cS}$.

However, $\norm{\cdot}{1/2,\cS}$ and $\norm{\cdot}{-1/2,\cS}$
are the norms for the trace and flux (across $\cS$) that make the norm
$\norm{\cdot}{\U}$ balanced, cf.~\eqref{norm_U}. Indeed, for $\beta=2\alpha=1/2$,
\begin{align*}
  \norm{\utrace}{1/2,\cS}
  &:=
  \inf \Big\{
    \bigl(\norm{w}{}^2 + \eps^{1/2}\norm{\nabla w}{}^2\bigr)^{1/2};\;
    w\in H^1_0(\Omega),\ w|_{\partial\el} = \utrace_\el\ \forall\el\in\TT
  \Big\},
  \\
  \norm{\sigmatrace}{-1/2,\cS}
  &:=
  \inf \Big\{
    \bigl(\norm{\qq}{}^2 + \eps\norm{\div\qq}{}^2\bigr)^{1/2};\;
    \qq\in \HH(\div,\Omega),\ \qq\cdot\n_\el|_{\partial\el} = \sigmatrace_\el\ \forall\el\in\TT
  \Big\}
\end{align*}
are trace norms that are induced by
$\bigl(\norm{u}{}^2+\eps^{1/2}\norm{\nabla u}{}^2 + \eps^{3/2}\norm{\Delta u}{}^2\bigr)^{1/2}$,
the balanced norm proposed in \cite{LinS_12_BFE} (note that $\sigmatrace$ is the normal trace of
$\ssigma=\eps^\alpha\nabla u$).

Now, using the stronger norm $\norm{\cdot}{1/2,\cS,b}$ for $\ub$ in $\norm{\cdot}{\U}$ means
that we would not be able to prove a robust best approximation result for problems
with typical boundary layers (in this paper we do not study approximation properties anyway).
But it is clear that we do not get rid of the sub-optimality
in the estimate \eqref{upper} (the factor $\eps^{-1/4}$ in front of $\norm{\ub}{1/2,\cS}$)
by simply defining different trace norms.
\end{remark}

\begin{lemma}\label{lem:inj}
  Let $\uu\in \U$ with $b(\uu,\vv)=0$ for all $\vv=(\ttau,\mu,v)\in \V$. Then $\uu=0$.
\end{lemma}
\begin{proof}
  Testing with functions from $C_0^\infty(\el)$ in~\eqref{b:1} and~\eqref{b:2} shows
  that $u\in H^1(\TT)$ and $\ssigma\in \HH(\div,\TT)$ with
  \begin{align}\label{lem:inj:1}
    \pwnabla u = \eps^{-\alpha}\ssigma \quad\text{ and }\quad \pwdiv\ssigma=\rho.
  \end{align}
  Integrating~\eqref{b:1} and~\eqref{b:2} by parts and using~\eqref{lem:inj:1} shows
  $u|_{\partial\el}=\ua_\el$ and $\ssigma\cdot\n|_{\partial\el} = \sigmaa|_{\partial\el}$
  for all $\el\in\TT$.
  Hence, $\ssigma\in \HH(\div,\Omega)$ and $u\in H^1(\Delta,\Omega)\cap H^1_0(\Omega)$.
  In particular, \eqref{lem:inj:1} then reads
  \begin{align}\label{lem:inj:2}
    \nabla u = \eps^{-\alpha}\ssigma \quad\text{ and }\quad \div\ssigma=\rho.
  \end{align}
  Since to $u\in H^1_0(\Omega)$ and $\ub=0$ on $\partial\Omega$,
  the definition of dualities on $\cS$ shows that
  $\dual{\sigmab}{u}=\dual{\ub}{\nabla u\cdot\n}=0$. Moreover, since $u\in H^1(\Delta,\Omega)$, we can
  choose $v=u$ in~\eqref{b:3}. Then, using identities~\eqref{lem:inj:2}, we obtain
  \begin{align*}
     \Bigl(\eps^{1-2\alpha}+\eps^{\beta-2\alpha}\Bigr) \vdual{\ssigma}{\ssigma}
     + \vdual{u}{u} + \eps^{1-2\alpha+\beta}\vdual{\rho}{\rho}
     &= 0,
  \end{align*}
  that is, $\ssigma$, $u$, and $\rho$ vanish. It remains to show that $\ub$ and $\sigmab$ vanish as well.
  Taking into account the results obtained so far we are left with the relation
  \begin{equation}\label{pf:inj:1}
     \eps^{1-\alpha}\dual{\sigmab}{v} + \eps^{\beta} \dual{\ub}{\pwnabla v\cdot\n} = 0
     \quad\forall v\in H^1(\Delta,\TT).
  \end{equation}
  Let $\ubext\in H^1_0(\Omega)$ be the extension of $\ub$ which is piecewise harmonic,
  i.e., on any $\el\in\TT$, $\ubext$ extends $\ub_\el$ harmonically onto $\el$.
  Then, integration by parts reveals that
  \[
     \dual{\ub_\el}{\nabla v\cdot\n_\el}_{\partial\el}
     =\dual{\nabla\ubext\cdot\n_\el}{v}_{\partial\el}
     \quad\forall v\in H^1(\Delta,\el)
     \quad\text{with}\quad \Delta v=0\ \text{on}\ \el.
  \]
  Therefore, \eqref{pf:inj:1} shows that
  \[
       \eps^{1-\alpha}\dual{\sigmab_\el}{v}_{\partial\el}
     + \eps^{\beta} \dual{\nabla\ubext\cdot\n_\el}{v}_{\partial\el}
     = 0
     \quad\forall v\in H^1(\el)\quad\text{with}\quad \Delta v=0\ \forall\el\in\TT.
  \]
  We conclude that $\nabla\ubext\cdot\n_\el=-\eps^{1-\alpha-\beta}\sigmab_\el$ on $\partial\el$
  for any $\el\in\TT$, i.e.,
  the normal derivatives of $\ubext$ across element boundaries do not jump
  (note that $\n_\el=-\n_{\el'}$ and $\sigmab_\el=-\sigmab_{\el'}$ on $\partial\el\cap\partial\el'$
  for neighboring elements $\el,\el'\in\TT$).
  Therefore, the piecewise harmonic function $\ubext$ is harmonic on $\Omega$.
  Since $\ubext=\ub=0$ on $\partial\Omega$ it follows that $\ubext=0$, $\ub=0$, and $\sigmab=0$.
\end{proof}
\subsection{Stability of the adjoint problem} \label{sec_stab}
In Section~\ref{sec_intro_dpg}, we have seen that a major part of DPG analysis deals with the stability of the adjoint problem.
As is standard in DPG theory (cf.~\cite{DemkowiczG_11_ADM,DemkowiczH_13_RDM}), this stability analysis
is split into several parts and combined by the superposition principle (or simply
the triangle inequality). In the following lemma we analyze the global inhomogeneous adjoint problem
for continuous functions, and Lemma~\ref{lem:adj:hom:aux} provides a technical stability result
for an intermediate homogeneous problem. Then, in Lemma~\ref{lem:adj:hom}, the homogeneous adjoint problem
with discontinuous functions is analyzed. All three results are combined in Corollary~\ref{cor}
and provide the remaining estimate used in the proof of Theorem~\ref{thm}.

\begin{lemma}\label{lem:adj:inhom}
  For $\alpha, \beta\in [0,1]$ with $\alpha+\beta\le 1$ and data
  $F,H\in L_2(\Omega)$, $\GG\in\LL_2(\Omega)$, there exists
  $(\ttau_1,\mu_1,v_1)\in \HH(\div,\Omega)\times H^1_0(\Omega)\times H^1_0(\Delta,\Omega)$ satisfying
  \begin{subequations}\label{lem:adj:inhom:eq}
    \begin{align}
      \div\ttau_1 + v_1 &= F\quad\text{ in } \Omega,\label{lem:adj:inhom:eq:1}\\
      \nabla\mu_1 + (\eps^{1-\alpha} + \eps^{\beta-\alpha})\nabla v_1 + \eps^{-\alpha}\ttau_1 &= \GG
      \quad\text{ in } \Omega,\label{lem:adj:inhom:eq:2}\\
      \eps^{1-\alpha+\beta}\Delta v_1 + \mu_1 &= H \quad\text{ in } \Omega,\label{lem:adj:inhom:eq:3}
    \end{align}
  \end{subequations}
  with
  \begin{align}\label{lem:adj:inhom:stab}
    \eps^{(1+\beta)/2}\norm{\Delta v_1}{} +
    \eps^{\beta/2}&\norm{\nabla v_1}{} + \norm{v_1}{} +
    \norm{\div\ttau_1}{}
    \lesssim \norm{F}{} + \eps^{\alpha-\beta/2}\norm{\GG}{}
    + \eps^{\alpha-(1+\beta)/2}\norm{H}{}
  \end{align}
  and
  \begin{align}\label{lem:adj:inhom:stab2}
    \eps^{-1/2}\norm{\mu_1}{} + \norm{\nabla\mu_1}{}
    + \eps^{-\alpha}\norm{\ttau_1}{}
    &\lesssim \eps^{\beta/2-\alpha}\norm{F}{} + \norm{\GG}{} +
    \eps^{-1/2}\norm{H}{}.
  \end{align}
\end{lemma}
\begin{proof}
  We construct a solution $(\ttau_1,\mu_1,v_1)$ of \eqref{lem:adj:inhom:eq} by first
  defining $v_1\in H^1_0(\Delta,\Omega)$ as the solution to a variational problem. We then proceed to
  select $\mu_1$ by \eqref{lem:adj:inhom:eq:3}, deduce that $\mu_1\in H^1_0(\Omega)$,
  define $\ttau_1$ by \eqref{lem:adj:inhom:eq:2}, prove that $\ttau_1\in\HH(\div,\Omega)$,
  and eventually show that $v_1$ and $\ttau_1$ satisfy \eqref{lem:adj:inhom:eq:1}.

  Now, the variational definition of $v_1$ is to be the solution of
  \begin{align}\label{lem:adj:inhom:b1}
    \text{find } v\in H^1_0(\Delta,\Omega) \text{ such that }
    b_A(v,w) = \ell_A(w) \text{ for all } w\in H^1_0(\Delta,\Omega),
  \end{align}
  where
  \begin{align*}
    b_A(v,w) &:= \eps^{1+\beta}\vdual{\Delta v}{\Delta w}
    + (\eps+\eps^{\beta})\vdual{\nabla v}{\nabla w}
    + \vdual{v}{w},\\
    \ell_A(w) &:=
    \vdual{F}{w} + \eps^\alpha\vdual{\GG}{\nabla w} + \eps^\alpha\vdual{H}{\Delta w}.
  \end{align*}
  If we equip $H^1_0(\Delta,\Omega)$ with the norm
  $\enorm{\cdot}_A^2 := \eps^{1+\beta}\norm{\Delta \cdot}{}^2
    + (\eps+\eps^{\beta})\norm{\nabla \cdot}{}^2
    + \norm{\cdot}{}^2$,
  then $b_A$ is continuous and elliptic with both constants being $1$,
  and $\ell_A$ is continuous with bound
  $\bigl(\norm{F}{}^2 + \eps^{2\alpha}(\eps+\eps^\beta)^{-1}\norm{\GG}{}^2
  + \eps^{2\alpha-1-\beta}\norm{H}{}^2\bigr)^{1/2}$. With the Lax-Milgram lemma we conclude
  that~\eqref{lem:adj:inhom:b1} has a unique solution $v_1\in H^1_0(\Delta,\Omega)$ with
  \begin{align}\label{lem:adj:inhom:v}
    \enorm{v_1}_A \lesssim \norm{F}{} + \eps^{\alpha-\beta/2}\norm{\GG}{}
    + \eps^{\alpha-(1+\beta)/2}\norm{H}{}.
  \end{align}
  We define $\mu_1\in L_2(\Omega)$ by equation~\eqref{lem:adj:inhom:eq:3} and conclude that
  \begin{align}\label{lem:adj:inhom:mu-L2}
    \eps^{-1/2}\norm{\mu_1}{}
    \lesssim
    \eps^{\beta/2-\alpha}\norm{F}{} + \norm{\GG}{} +
    \eps^{-1/2}\norm{H}{}.
  \end{align}
  We continue to show that indeed $\mu_1\in H^1_0(\Omega)$ with the desired $H^1(\Omega)$ bound.
  To this end define the norm 
  $\enorm{\cdot}_B^2:=\eps^{-1}\norm{\cdot}{}^2 + \norm{\nabla\cdot}{}^2$.
  By definition of $\mu_1$ and $v_1$ we have for all $\varphi\in C_0^\infty(\Omega)$ the identity
  \begin{align*}
    \lefteqn{
    \vdual{\mu_1}{(\Delta - \eps^{-1})\varphi} = \vdual{H}{\Delta\varphi}
    - \eps^{1-\alpha+\beta}\vdual{\Delta v_1}{\Delta\varphi} - \eps^{-1}\vdual{\mu_1}{\varphi}}
    \\
    &=-\eps^{-\alpha}\vdual{F}{\varphi} - \vdual{\GG}{\nabla\varphi}
    + (\eps^{1-\alpha} + \eps^{\beta-\alpha}) \vdual{\nabla v_1}{\nabla \varphi}
    + \eps^{-\alpha}\vdual{v_1}{\varphi} - \eps^{-1}\vdual{\mu_1}{\varphi},
  \end{align*}
  such that, using~\eqref{lem:adj:inhom:v} and~\eqref{lem:adj:inhom:mu-L2}, we arrive at
  \begin{align}\label{lem:adj:inhom:mu-H1}
    \sup_{\varphi\in C_0^\infty(\Omega)}
    \frac{\abs{\vdual{\mu_1}{(-\Delta+\eps^{-1})\varphi}}}{\enorm{\varphi}_B}
    &\lesssim
    \eps^{\beta/2-\alpha}\norm{F}{} + \norm{\GG}{} + \eps^{-1/2}\norm{H}{}.
  \end{align}
  The Lax-Milgram lemma shows that the operator $-\Delta + \eps^{-1}$ is an isomorphism from the Hilbert space
  $\cH := \left(H^1_0(\Omega), \enorm{\cdot}_B\right)$ to its dual $\cH'$ and that the continuity constants
  of $-\Delta + \eps^{-1}$ and its inverse do not depend on $\eps$. Furthermore, as $C_0^\infty(\Omega)$
  is dense in $\cH$, we conclude that $(-\Delta+\eps^{-1})(C_0^\infty(\Omega))$ is dense in $\cH'$.
  Hence, with~\eqref{lem:adj:inhom:mu-H1},
  \begin{align*}
    \enorm{\mu_1}_B = \sup_{g \in \cH'} \frac{\abs{g(\mu_1)}}{\norm{g}{\cH'}}
    &\simeq
    \sup_{\varphi\in C_0^\infty(\Omega)}
    \frac{\abs{\vdual{\mu_1}{(-\Delta+\eps^{-1})\varphi}}}{\enorm{\varphi}_B}
    \lesssim
    \eps^{\beta/2-\alpha}\norm{F}{} + \norm{\GG}{} + \eps^{-1/2}\norm{H}{}.
  \end{align*}
  This shows that $\mu_1\in H^1_0(\Omega)$ with the desired bound.
  Finally, we define $\ttau_1 \in \LL_2(\Omega)$ by equation~\eqref{lem:adj:inhom:eq:2}.

  We now show that $\ttau_1\in\HH(\div,\Omega)$, and that $\ttau_1$ and $v_1$
  satisfy \eqref{lem:adj:inhom:eq:1}.
  Let $\varphi\in C^\infty_0(\Omega)$ be given. We test \eqref{lem:adj:inhom:eq:2} with $\eps^\alpha\nabla\varphi$,
  \eqref{lem:adj:inhom:eq:3} with $\eps^\alpha\Delta\varphi$, and integrate by parts the latter equation.
  Summation of both equations yields
  \begin{align*}
      \eps^{1+\beta}\vdual{\Delta v_1}{\Delta\varphi}
    + (\eps+\eps^{\beta})\vdual{\nabla v_1}{\nabla\varphi}
    + \vdual{\ttau_1}{\nabla\varphi}
    &=
      \eps^\alpha \vdual{\GG}{\nabla\varphi}
    + \eps^\alpha \vdual{H}{\Delta\varphi}.
  \end{align*}
  Taking into account the variational definition \eqref{lem:adj:inhom:b1} of $v_1$, this relation reduces to
  \begin{align*}
     \vdual{\ttau_1}{\nabla\varphi} - \vdual{v_1}{\varphi} &= -\vdual{F}{\varphi}.
  \end{align*}
  We conclude that $\ttau_1\in \HH(\div,\Omega)$, and that $\ttau_1$, $v_1$ satisfy \eqref{lem:adj:inhom:eq:1}.
  The bounds for $\ttau_1$ follow from relations~\eqref{lem:adj:inhom:eq:1},~\eqref{lem:adj:inhom:eq:2}
  and the previous bounds for $v_1$ and $\mu_1$.
\end{proof}
\begin{remark} \label{rem_alpha_beta}
Our aim is to control the unknown functions $u$ and $\ssigma$ in $L_2$ in a robust way. That is,
principal objective is to bound these parts of the $\U$-norm by the energy norm with a constant
that is independent of $\eps$. By DPG-theory, this bound is equivalent to the uniform stability of
the adjoint problem \eqref{lem:adj:inhom:eq} with right-hand side functions $F$ and $\GG$ taken in $L_2$.
In fact, the robust control is down to the constant $C_{\rm infsup}$ in \eqref{norm_E}
which comes from the inf-sup property \eqref{eq:bb:2}. This latter property is equivalent to a
robust bound \eqref{normeq}, the stability of the adjoint problem.

According to the upper bounds in \eqref{lem:adj:inhom:stab} and \eqref{lem:adj:inhom:stab2}, this is only
achievable if $\alpha-\beta/2=0$. Furthermore, the lower bounds in these estimates control the test norm
in $\V$ only if $\alpha\le 1/4$. From the point of view of boundedness of the bilinear form $b(\cdot,\cdot)$
(cf.~Lemma~\ref{lem:b:stab}) we want to select $\alpha$ and $\beta$ as large as possible. Therefore,
the natural selection is $\alpha=1/4$, $\beta=1/2$ and our method provides robust control of the variables
$u$, $\ssigma=\eps^{1/4}\nabla u$, and $\eps^\beta\rho=\eps^{3/4}\Delta u$,
cf.~\eqref{norm_U} for the weighting of $\rho$.
According to \cite{LinS_12_BFE}, in the presence of boundary layers,
precisely these $\eps$-weightings guarantee that the three unknowns
have comparable $L_2$-norms when $\eps\to 0$. Our DPG analysis with robustness as
objective leads to the very weightings without any approximation theory for specific solutions.
\end{remark}
\begin{lemma}\label{lem:adj:hom:aux}
  Suppose that $(\llambda,w)\in \HH(\div,\TT)\times H^1(\TT)$ satisfy
  \begin{subequations}\label{lem:adj:hom:aux:eq:1}
    \begin{align}
      \llambda + \nabla w &= 0\label{lem:adj:hom:aux:eq:1a},\\
      \div \llambda + \eps^{-\beta}w &=0\label{lem:adj:hom:aux:eq:1b}
    \end{align}
  \end{subequations}
  on any $\el\in\TT$. Then it holds that
  \begin{align*}
    \norm{\nabla w}{} = \norm{\llambda}{}
    &\lesssim
    \eps^{-\beta}\norm{w}{-1} + \norm{\jump{\llambda\cdot\n}}{-1/2,\cS'}
                                     + \norm{\jump{w}}{1/2,\cS'}
  \end{align*}
  and
  \begin{align*}
    \norm{w}{} = \eps^{\beta} \norm{\pwDelta w}{}
    &\lesssim
      (\eps^\beta+\eps^{\alpha+\beta/2})\norm{\jump{\llambda\cdot\n}}{-1/2,\cS'}
    + \eps^{\beta/2} \norm{\jump{w}}{1/2,\cS'}.
  \end{align*}
\end{lemma}
\begin{proof}
  We follow the ideas used in \cite[Lemma 4.4]{DemkowiczG_11_ADM} but have to consider the parameter $\eps$.
  In three dimensions we use the Helmholtz decomposition $\llambda = \nabla\psi + \ccurl\zz$
  with $\psi\in H^1_0(\Omega)$ and $\zz\in \HH(\ccurl,\Omega)$.
  It follows that
  $\norm{\nabla\psi}{}^2+\norm{\ccurl\zz}{}^2\le\norm{\llambda}{}^2$
  and, by the Poincar\'e-Friedrichs inequality, $\norm{\psi}{}\lesssim \norm{\llambda}{}$.
  By definition of the trace norms we also bound
  \begin{align*}
      &\norm{\psi}{1/2,\cS}
      \lesssim
      (1 + \eps^\alpha)\norm{\llambda}{} \simeq \norm{\llambda}{},
      \qquad
      \norm{\n\cdot\ccurl\zz}{-1/2,\cS} \le \norm{\ccurl\zz}{} \le \norm{\llambda}{}.
  \end{align*}
  Piecewise integration by parts and~\eqref{lem:adj:hom:aux:eq:1} yield
  \begin{align*}
    \vdual{\llambda}{\llambda} =
    \eps^{-\beta} \vdual{w}{\psi} + \dual{\llambda\cdot\n}{\psi} - \dual{w}{\n\cdot\ccurl\zz}
  \end{align*}
  so that the previous bounds and dualities prove the first assertion.

  Now define $\psi\in H^1_0(\Omega)$ to be the weak solution of $-\eps^\beta \Delta\psi + \psi=-w$ such that
  $\norm{\psi}{} + \eps^{\beta/2}\norm{\nabla\psi}{}
  + \eps^\beta\norm{\Delta\psi}{} \lesssim \norm{w}{}$.
  By definition of the trace norms it holds that
  \begin{align*}
    \begin{split}
      \norm{\psi}{1/2,\cS} &\lesssim (1 + \eps^{\alpha-\beta/2})\norm{w}{},\\
      \norm{\nabla\psi\cdot\n}{-1/2,\cS} &\lesssim (\eps^{-\beta/2} + 1)\norm{w}{}
                                   \simeq \eps^{-\beta/2}\norm{w}{}.
    \end{split}
  \end{align*}
  Piecewise integrating by parts twice, and using \eqref{lem:adj:hom:aux:eq:1}, we obtain
  \begin{align*}
    \vdual{w}{w} &= -\eps^\beta\vdual{\pwnabla w}{\nabla \psi} - \vdual{w}{\psi}
    +\eps^\beta\dual{\nabla\psi\cdot\n}{w}
    = \eps^\beta\dual{\llambda\cdot\n}{\psi} + \eps^\beta\dual{\nabla\psi\cdot\n}{w}.
  \end{align*}
  The previous estimates for the trace norms of $\psi$ show the second assertion.

  In two dimensions one uses the Helmholtz decomposition $\llambda=\nabla\psi + (-\partial_2z,\partial_1z)$
  with scalar potential $z\in H^1(\Omega)$. Then the assertions follow as before.
\end{proof}
\begin{lemma}\label{lem:adj:hom}
  Let $\beta=2\alpha=1/2$ and
  $(\ttau_0,\mu_0,v_0)\in \HH(\div,\TT)\times H^1(\TT)\times H^1(\Delta,\TT)$ be a solution of
  \begin{subequations}\label{lem:adj:hom:eq:1}
    \begin{align}
      \pwdiv\ttau_0 + v_0 &= 0,\label{lem:adj:hom:eq:1a}\\
      \pwnabla\mu_0 + (\eps^{1-\alpha} + \eps^{\beta-\alpha})\pwnabla v_0 + \eps^{-\alpha}\ttau_0 &= 0,
      \label{lem:adj:hom:eq:1b}\\
      \eps^{1-\alpha+\beta}\pwDelta v_0 + \mu_0 &= 0\label{lem:adj:hom:eq:1c}
    \end{align}
  \end{subequations}
  on any $\Omega$. Then, with
  \[\begin{split}
     &\norm{\jump{\ttau_0,\mu_0,v_0}}{} :=\\
     &\qquad
      \norm{\jump{\pwnabla v_0\cdot\n}}{-1/2,\cS'}
    + \eps^{-3/4}\norm{\jump{\ttau_0\cdot\n}}{-1/2,\cS'}
    + \eps^{-3/4}\norm{\jump{\mu_0}}{1/2,\cS'}
    + \eps^{-1/2}\norm{\jump{v_0}}{1/2,\cS'}
  \end{split}\]
  it holds that
  \begin{align*}
      \norm{v_0}{}
    + \norm{\pwdiv\ttau_0}{}
    &\lesssim
    \eps^{3/4} \norm{\jump{\ttau_0,\mu_0,v_0}}{}
    \\
      \eps^{3/4}\norm{\pwDelta v_0}{}
    + \eps^{-1/4} \norm{\ttau_0}{}
    + \eps^{-1/2}\norm{\mu_0}{}
    &\lesssim
    \eps^{1/2} \norm{\jump{\ttau_0,\mu_0,v_0}}{}
    \\
      \eps^{1/4}\norm{\pwnabla v_0}{}
    + \norm{\pwnabla\mu_0}{}
    &\lesssim
    \hspace{1.6em}\norm{\jump{\ttau_0,\mu_0,v_0}}{}.
  \end{align*}
  In particular, taking the largest upper bound, we have the estimate
  \begin{align*}
     \norm{(\ttau_0,\mu_0,v_0)}{\V} &\lesssim \norm{\jump{\ttau_0,\mu_0,v_0}}{}.
  \end{align*}
\end{lemma}

\begin{proof}
  Define 
  \begin{align} \label{def_w_lambda}
    w := -\eps\pwDelta v_0 + v_0 \quad \text{ and } \quad
    \llambda := \eps^{1/2}\pwnabla v_0 + \eps^{-1/2}\ttau_0.
  \end{align}
  We start by bounding $\norm{\pwnabla v_0}{}$. 
  As previously, we use a Helmholtz decomposition $\pwnabla v_0 = \nabla\psi + \ccurl\zz$
  with $\psi\in H^1_0(\Omega)$ and $\zz\in \HH(\ccurl,\Omega)$.
  Analogously as in the proof of Lemma~\ref{lem:adj:hom:aux}, the definitions of trace norms and stability
  of the Helmholtz decomposition show that
  \begin{align*}
      \norm{\psi}{1/2,\cS} &\lesssim
      \norm{\pwnabla v_0}{},\qquad
      \norm{\n\cdot\ccurl\zz}{-1/2,\cS}
      \le \norm{\pwnabla v_0}{}.
  \end{align*}
  Then, piecewise integration by parts and the definition of $w$ yield
  \begin{align*}
    \vdual{\pwnabla v_0}{\pwnabla v_0} =
    \eps^{-1} \vdual{w-v_0}{\psi} + \dual{\pwnabla v_0\cdot\n}{\psi} + \dual{v_0}{\n\cdot\ccurl\zz}
  \end{align*}
  so that, using the stability of the Helmholtz decomposition
  $\norm{\psi}{H^1_0(\Omega)}\lesssim \norm{\pwnabla v_0}{}$,
  \begin{align} \label{pfxx1}
    \norm{\pwnabla v_0}{}
    &\lesssim
      \eps^{-1} \norm{w-v_0}{-1}
    + \norm{\jump{\pwnabla v_0\cdot\n}}{-1/2,\cS'}
    + \norm{\jump{v_0}}{1/2,\cS'}
    \nonumber\\
    &\lesssim
    \eps^{-1} \norm{w}{} + \eps^{-1} \norm{v_0}{}
    + \norm{\jump{\pwnabla v_0\cdot\n}}{-1/2,\cS'} + \norm{\jump{v_0}}{1/2,\cS'}.
  \end{align}
  Now define $\psi\in H^1_0(\Omega)$ to be the weak solution of $-\eps\Delta\psi + \psi=-v_0$ such that
  $\norm{\psi}{} + \eps^{1/2}\norm{\nabla\psi}{} + \eps\norm{\Delta\psi}{}
   \lesssim \norm{v_0}{}$.
  By definition of the trace norms we find
  \begin{align*}
      \norm{\psi}{1/2,\cS} &\lesssim (1 + \eps^{-1/4})\norm{v_0}{},\qquad
      \norm{\nabla\psi\cdot\n}{-1/2,\cS} \lesssim \eps^{-1/2} \norm{v_0}{}.
  \end{align*}
  Twice integrating piecewise by parts, and using the definition of $w$, we obtain
  \begin{align*}
    \vdual{v_0}{v_0}
    &=
     \vdual{\eps\pwDelta v_0-v_0}{\psi}
    +\eps\dual{\nabla\psi\cdot\n}{v_0} - \eps\dual{\pwnabla v_0\cdot\n}{\psi}\\
    &=
    -\vdual{w}{\psi} +\eps\dual{\nabla\psi\cdot\n}{v_0} - \eps\dual{\pwnabla v_0\cdot\n}{\psi}.
  \end{align*}
  The previous estimates for the trace norms of $\psi$ show that
  \begin{align} \label{pfxx2}
    \norm{v_0}{} \lesssim 
    \norm{w}{} + \eps^{1/2} \norm{\jump{v_0}}{1/2,\cS'}
    + \eps^{3/4} \norm{\jump{\pwnabla v_0\cdot\n}}{-1/2,\cS'}.
  \end{align}
  Furthermore, by the definition of $w$,
  \begin{align} \label{pfxx3}
  \lefteqn{
    \norm{\pwDelta v_0}{}
    \le
    \eps^{-1}\norm{w-v_0}{}
  }\nonumber\\
    &\lesssim
    \eps^{-1}\norm{w}{} + \eps^{-1/2} \norm{\jump{v_0}}{1/2,\cS'}
    + \eps^{-1/4} \norm{\jump{\pwnabla v_0\cdot\n}}{-1/2,\cS'}.
  \end{align}
  Relations \eqref{lem:adj:hom:eq:1} show that $\llambda$ and $w$ satisfy \eqref{lem:adj:hom:aux:eq:1}
  so that we can use the bounds of Lemma~\ref{lem:adj:hom:aux}. By \eqref{lem:adj:hom:eq:1c},
  the definition of $w$ yields $w=\eps^{-1/4}\mu_0+v_0$, cf.~\eqref{def_w_lambda}.
  Using this representation and the definition \eqref{def_w_lambda} of $w$, $\llambda$, the bound
  by Lemma~\ref{lem:adj:hom:aux} gives
  \begin{align} \label{xxx}
     \norm{w}{}
     &\lesssim
       \eps \norm{\jump{\pwnabla v_0\cdot\n}}{-1/2,\cS'}
     + \norm{\jump{\ttau_0\cdot\n}}{-1/2,\cS'}
     + \norm{\jump{\mu_0}}{1/2,\cS'}
     + \eps^{1/4} \norm{\jump{v_0}}{1/2,\cS'}.
  \end{align}
  Using this estimate in \eqref{pfxx2} then yields
  \begin{align} \label{pfxx4}
     \norm{v_0}{}
     &\lesssim
       \eps^{3/4} \norm{\jump{\pwnabla v_0\cdot\n}}{-1/2,\cS'}
     + \norm{\jump{\ttau_0\cdot\n}}{-1/2,\cS'}
     + \norm{\jump{\mu_0}}{1/2,\cS'}
     + \eps^{1/4} \norm{\jump{v_0}}{1/2,\cS'},
  \end{align}
  which is the assertion for $v_0$.
  Correspondingly, from~\eqref{pfxx1}, \eqref{xxx} and \eqref{pfxx4}, we deduce the statement
  for $\norm{\pwnabla v_0}{}$, and \eqref{pfxx3} and \eqref{xxx} prove the assertion
  for $\norm{\pwDelta v_0}{}$. We have thus provided bounds for all terms depending on $v_0$.

  We continue with $\ttau_0$.
  By the definition \eqref{def_w_lambda} of $\llambda$,
  $\ttau_0=\eps^{1/2}\llambda - \eps\pwnabla v_0$. The latter term has been dealt with and
  Lemma~\ref{lem:adj:hom:aux} together with previous estimates bound $\norm{\llambda}{}$.
  By relation \eqref{lem:adj:hom:eq:1a}, $\norm{\pwdiv\ttau_0}{}$ can be estimated through
  \eqref{pfxx4}.

  It remains to consider the norms of $\mu_0$. By \eqref{lem:adj:hom:eq:1c},
  $\eps^{-1/2}\norm{\mu_0}{}=\eps^{3/4}\norm{\pwDelta v_0}{}$ and this term
  has already been analyzed. To estimate
  $\norm{\pwnabla\mu_0}{}$, by \eqref{lem:adj:hom:eq:1b} it is enough to
  bound $\eps^{-1/4}\norm{\ttau_0}{}$ and $\eps^{1/4}\norm{\pwnabla v}{}$,
  which we have done. This finishes the proof of the lemma.
\end{proof}
Let us combine the findings from Lemmas~\ref{lem:adj:inhom} and \ref{lem:adj:hom}.

\begin{cor} \label{cor}
Set $\beta=1/2$ and $\alpha=1/4$. Then it holds that
\begin{align*}
\lefteqn{
   \norm{\vv}{\V}
   \lesssim
     \norm{\vv}{\V,\opt}
}
   \\
   &+ \eps^{-1/2} \norm{\jump{\eps^{1/2}\pwnabla v\cdot\n}}{-1/2,\cS'}
    + \eps^{-3/4}\norm{\jump{\ttau\cdot\n}}{-1/2,\cS'}
    + \eps^{-3/4}\norm{\jump{\mu}}{1/2,\cS'}
    + \eps^{-5/4}\norm{\jump{\eps^{3/4} v}}{1/2,\cS'}.
\end{align*}
for any $\vv=(\ttau,\mu,v)\in \V$.
\end{cor}

\begin{proof}
The proof is standard, but is indicated for ease of reading.
For a given $\vv=(\ttau,\mu,v)\in \V$ define $\vv_1=(\ttau_1,\mu_1,v_1)\in \V$
as a solution of \eqref{lem:adj:inhom:eq} with
\begin{align*}
      F &:= \pwdiv\ttau + v,\quad
      \GG := \pwnabla\mu + (\eps^{1-\alpha} + \eps^{\beta-\alpha})\pwnabla v + \eps^{-\alpha}\ttau,\quad
      H := \eps^{1-\alpha+\beta}\pwDelta v + \mu        
\end{align*}
in $\Omega$. Then, $\vv_0=(\ttau_0,\mu_0,v_0):=\vv-\vv_1$ solves \eqref{lem:adj:hom:eq:1} and
the estimates from Lemmas~\ref{lem:adj:inhom} and \ref{lem:adj:hom} prove that
\begin{align*}
\lefteqn{
   \norm{\vv}{\V}
   \le \norm{\vv_1}{\V} + \norm{\vv_0}{\V}
   \lesssim
   \norm{F}{} + \norm{\GG}{} + \eps^{-1/2}\norm{H}{}
   +
}
   \\
   &
   \eps^{-1/2}\norm{\jump{\eps^{1/2}\pwnabla v_0\cdot\n}}{-1/2,\cS'}
   + \eps^{-3/4}\norm{\jump{\ttau_0\cdot\n}}{-1/2,\cS'}
   + \eps^{-3/4}\norm{\jump{\mu_0}}{1/2,\cS'}
   + \eps^{-5/4}\norm{\jump{\eps^{3/4} v_0}}{1/2,\cS'}.
\end{align*}
By construction of $\vv_1$,
\[
     \norm{\jump{\pwnabla v_1\cdot\n}}{-1/2,\cS'}
   = \norm{\jump{\ttau_1\cdot\n}}{-1/2,\cS'}
   = \norm{\jump{\mu_1}}{1/2,\cS'}
   = \norm{\jump{v_1}}{1/2,\cS'}
   = 0.
\]
Therefore the assertion follows with the characterization of the optimal test norm, cf.~\eqref{norm_Vopt}.
\end{proof}

\section{Numerical experiments}\label{sec_num}
We present several numerical experiments based on three different problems in two space dimensions.
The first problem, in~Subsection~\ref{exp1}, consists of a manufactured solution taken
from~\cite{LiuMSZ_09_TSS}. By means of this problem, we show that,
for smooth enough right-hand side $f$ of~\eqref{model}, our
method leads to best approximations in \textbf{balanced norms} for uniform and local mesh refinement.

For the second problem (Subsection~\ref{exp2}) we choose the right-hand side $f$ of~\eqref{model}
to have support only on a compact subset of the computational domain $\Omega$. This way, the problem
will exhibit inner layers which are not aligned with the mesh. We use adaptive mesh refinement
to show that our method \textbf{automatically resolves the layers}.
Note that DPG methods automatically provide a posteriori error estimates, cf.~\eqref{eq:est}.
In our case, Theorem~\ref{thm} combined with \eqref{eq:est} yields the robust control by
$\norm{B\uu_\hp-L}{V'}$ of the error in the balanced norm of the field variables.
As mentioned after relation \eqref{eq:est}, the residual $\norm{B\uu_\hp - L}{V'}$ is computable
locally due to the product structure of the test space $V$. Here, we do not consider the influence
of the approximation of the optimal test functions.

Furthermore, we will show
computed solutions with extremely small $\eps$ to demonstrate \textbf{robustness} of the approximations
in the sense that they are basically free of oscillations.

For the third problem (Subsection~\ref{exp3}) we choose a computational domain $\Omega$ with a 
re-entrant corner and a right-hand side $f$ of~\eqref{model} such that the solution $u$
exhibits singularities. Uniform mesh refinement will lead to sub-optimal convergence rates, but
adaptive mesh refinement will \textbf{recover optimal rates}.

To interpret our numerical results below one has to take into account the following three facts.
First, as stated in the introduction, our DPG analysis is based on the use of optimal test functions.
We did not analyze the effect of approximating these test functions. When considering
a singularly perturbed problem as the one under consideration, this discrepancy will have an
effect that increases when the perturbation parameter $\eps$ becomes smaller. Specifically,
one may lose robustness of the estimate in Theorem~\ref{thm} when using poor approximations
of optimal test functions (we do observe this).
Second, the DPG method (with optimal test functions) provides best approximations in the energy norm.
Since the domain of trace spaces (the skeleton) grows when meshes are refined one does not have
hierarchy of approximation spaces based on mesh refinement. This means that the error
may be not monotone (we do observe this in a preasymptotic range). Third, for small $\eps$
the solutions of reaction diffusion problems have strong boundary layers. In these cases
the primal unknown $u$ can be approximated well on coarse meshes but the flux $\ssigma$ and
the Laplacian $\rho$ can not. Therefore, for coarse meshes and comparatively small $\eps$,
the individual approximation errors from $u$, $\ssigma$ and $\rho$ constituting the balanced norm
can have different magnitudes. Indeed, the error in the balanced norm can be large for coarse
meshes. This is not a problem of the DPG method but an approximation property (once one accepts
the use of the balanced norm). It also does not contradict the balancedness of the norm which
holds for the exact solution and typical boundary layers.

We use triangular meshes $\TT$, and throughout $\#\TT$ denotes the number of triangles.
In all experiments we use the trial space $\U_\hp\subset \U$ defined by
\begin{align*}
  \U_\hp := P^0(\TT) \times \left[ P^0(\TT) \right]^2 \times P^0(\TT)
  \times S^1_0(\cS) \times S^1_0(\cS) \times P^0(\cS) \times P^0(\cS).
\end{align*}
The trial-to-test operator $\Theta=J^{-1}B$ needed for the computation
of optimal test functions is approximated using the discrete operator
$J_\hp:\; \V_\hp\to\V_\hp'$ with finite-dimensional space $\V_\hp\subset \V$ defined by
\begin{align*}
  \V_\hp := \left[ P^4(\TT) \right]^2 \times P^4(\TT) \times P^4(\TT).
\end{align*}
Here, $P^p(\TT)$ is the space of $\TT$-piecewise polynomials with degree at most $p$.
The basis for $P^p(\TT)$ is
based on Lobatto shape functions on the reference elements,
as defined in~\cite[Sections~2.2.2 and 2.2.3]{SolinSD_04_HOF}.
The space $S^1_0(\cS)$ is the space of globally continuous,
$\cS$-piecewise linear functions, and $P^0(\cS)$ is the space of $\cS$-piecewise constant functions.
For $H(\div)$-parts in our bilinear forms, we use the standard element map instead of the Piola transform.
Our choice of $\V_\hp$ is based on the analysis in~\cite{GopalakrishnanQ_14_APD}, where the authors consider
the Laplace equation.
Although their analysis is not directly applicable to the problem studied in the paper at hand,
we can use it to heuristically choose the approximation order of the test functions in our discretization.
In~\cite{GopalakrishnanQ_14_APD}, the authors show
that for a valid approximation $J_\hp$ of the Riesz operator $J$ for the Laplace equation, it suffices to
raise the polynomial degree of the trial space by the dimension $d$ of the physical space $\R^d$.
As our discretization amounts to a bi-Laplace equation in $\R^2$,
we raise the polynomial degree by $2\cdot d = 2\cdot2 = 4$.
A theoretical analysis of this additional approximation as in~\cite{GopalakrishnanQ_14_APD} is out of
the scope of this paper and is left for future research. In the experiment of Section~\ref{exp1}
below, we numerically investigate the influence of different orders of polynomial approximation for
the test functions.
The experiments were performed in C++. The inverse $\mathbf{J}^{-1}$ of the matrix
corresponding to the approximated Riesz operator $J_\hp$ is once and for all
computed block-wise with a Cholesky decomposition. The overall linear system is then written as
$\mathbf{B}^T \mathbf{J}^{-T} \mathbf{B} \mathbf{x} = \mathbf{B}^T \mathbf{J}^{-T} \mathbf{f}$
with $\mathbf{B}$ and $\mathbf{f}$ being the discretizations of $B$ and the linear functional $L$,
respectively. It is solved by conjugate gradients without preconditioning.

For adaptive mesh refinement with mesh sequence $\TT_\ell$,
we start with a coarse mesh $\TT_0$. In order to compute
the mesh $\TT_{\ell+1}$ from $\TT_\ell$,
we define, in accordance with~\eqref{eq:est}, the local error indicator
\begin{align*}
  \eta_\ell(\el)^2 := \norm{J_\hp|_\el^{-1}(L-B\uu_\ell)}{\V|_\el}^2
\end{align*}
for all $\el\in\TT_\ell$.
Here $\uu_\ell$ is the DPG solution on the mesh $\TT_\ell$.
We then mark a set of elements $\MM_\ell\subset\TT_\ell$ with minimal cardinality such that
$\sum_{\el\in\MM_\ell} \eta_\ell(\el)^2 \geq \theta \sum_{\el\in\TT_\ell}\eta_\ell(\el)^2$.
In all experiments, we choose $\theta = 0.75$, and for local mesh-refinement we use the so-called
\textit{Newest Vertex Bisection}, cf.~\cite{Baensch_91_LMR}.
\subsection{Problem with manufactured solution}\label{exp1}
The following example is taken from~\cite{LiuMSZ_09_TSS}:
\begin{align*}
  -\eps\Delta u + (1+x^2y^2 e^{xy/2})u = f \quad\text{ on }\Omega := (0,1)^2,
\end{align*}
where
\begin{align*}
  u(x,y) = &x^3(1+y^2) + \sin(\pi x^2) + \cos(\pi y/2)\\
  &(x+y)\left[ e^{-2x/\sqrt{\eps}} + e^{-2(1-x)/\sqrt{\eps}} + e^{-3y/\sqrt{\eps}} 
  + e^{-3(1-y)/\sqrt{\eps}} \right].
\end{align*}
Although in our analysis we assumed $c=1$ for the reaction coefficient in $-\eps\Delta u + cu=f$
and homogeneous Dirichlet boundary condition, our method can be extended in a simple way to cover more general cases.
The incorporation of arbitrary boundary values is done in a standard way by extending them
to the domain $\Omega$, while we deal with the case $c\neq1$ by testing with $v-\eps^\beta \Delta_\TT v / c$,
cf.~\cite{LinS_12_BFE}. In Fig.~\ref{fig:man}, we present the outcome of experiments
with different values of $\eps$ ranging from $1$ to $10^{-16}$ (indicated by different colors),
for uniform as well as adaptive mesh refinement (indicated by crosses and squares). A detailed legend is given in the lower left plot.
We plot one error (upper left) and three quotients of errors (lower left, upper and lower right).
\begin{itemize}
  \item \textbf{upper left:} This graph shows the squared energy error $\norm{\uu-\uu_\hp}{E}^2$ versus the number
    of triangles $\#\TT$. For uniform and adaptive refinement, we see an asymptotic behaviour of $\OO(\#\TT^{-1})$
    (which amounts to $\norm{\uu-\uu_\hp}{E}=\OO(h)$ in the uniform case) as soon as the boundary layers are resolved.
    This happens instantly for $\eps=1$, for smaller values of $\eps$ the error increases before it runs into the
    asymptotic regime, and in some cases the asymptotic range is not reached for the considered number of unknowns.
    Note that, for adaptive refinement, this increase happens faster (which means that adaptivity performs better).
  \item \textbf{upper right:} This graph shows the quotient
    $(\norm{u-u_\hp}{}^2 + \norm{\ssigma-\ssigma_\hp}{}^2 + \eps\norm{\rho-\rho_\hp}{}^2) / \norm{\uu-\uu_\hp}{E}^2$
    of the field variables in the balanced norm and the energy error versus the number of triangles $\#\TT$.
    According to Theorem~\ref{thm}, when using exact optimal test functions this quotient is bounded from
    above independently of $\eps$. We see that this is the case only for moderate $\eps$ or when meshes are sufficiently
    fine. Our explanation is that we have used only approximated optimal test functions.
    Sufficiently fine meshes that resolve boundary layers allow for good approximations
    of optimal test functions and then, the error ratio stabilizes independently of $\eps$.
    Again, this stabilization happens faster for the adaptive version.
  \item \textbf{lower left and lower right:} 
    We expect our method to deliver best approximations in balanced norms.
    However, this property does not mean that the individual terms of the error are balanced
    uniformly in $\eps$ and $\TT$. For a coarse mesh relative to $\eps$, $u$ can be approximated well
    whereas $\ssigma$ and $\rho$ have large values in the layers and their approximations will be worse.
    This is confirmed in the lower left and right plots of Fig.~\ref{fig:man}. There
    we plot the quotients $\norm{u-u_\hp}{}^2 / \norm{\ssigma-\ssigma_\hp}{}^2$
    and $\norm{u-u_\hp}{}^2 / (\eps \norm{\rho-\rho_\hp}{}^2)$, respectively.
    We observe that for $\eps\in \{10^{0,-4,-6}\}$ we reach the asymptotic range, at least for the adaptive versions, where
    the ratios are of order $O(1)$. For smaller $\eps$ (and $\eps=10^{-6}$ with uniform meshes) we have a clear dominance
    of the approximation errors of $\ssigma$ and $\rho$. Eventually, when meshes are fine enough, one sees a stabilization
    but not yet of the order $O(1)$. Note that this stabilization happens faster for the adaptive version.
    Therefore, the observed behavior can be explained by the approximation properties of spaces for components with different layers
    (and as before, the approximation of optimal test functions will have an effect).
    Our DPG scheme is not designed to provide the best approximation of $u$ in $L_2$,
    but to minimize the energy error $\norm{\uu-\uu_\hp}{E}$ which contains the balanced
    norm of the errors of all field variables. Once the layers are resolved, we expect
    the quotients $\norm{u-u_\hp}{}^2 / \norm{\ssigma-\ssigma_\hp}{}^2$ and
    $\norm{u-u_\hp}{}^2 / (\eps \norm{\rho-\rho_\hp}{}^2)$ to be of the order $O(1)$ (though we have not proved this).
\end{itemize}

The influence of the polynomial order for the approximation of the test functions is shown in Fig.~\ref{fig:man:testfunapx}.
We approximate the test space $V$ by the space
\begin{align*}
  \V_r := \left[ P^r(\TT) \right]^2 \times P^r(\TT) \times P^r(\TT)
\end{align*}
with $r=0,\dots, 6$, and plot the $L_2$ errors of the field variables $u$, $\ssigma$, and $\rho$.
While $r=0,1$ are obviously not sufficient, it seems that $r=2$ already yields optimal convergence rates.

\begin{sidewaysfigure}[htb]
  \psfrag{energy error^2}[cc][cc]{\scriptsize$\norm{\uu-\uu_\hp}{E}^2$}
  \psfrag{L2error^2}[cc][cc]{\scriptsize$\norm{u-u_\hp}{}^2$}
  \psfrag{(e^1/4 H1-error)^2}[cc][cc]{\scriptsize$\norm{\ssigma-\ssigma_\hp}{}^2$}
  \psfrag{(e^3/4 H2-error)^2}[cc][cc]{\scriptsize$\eps\norm{\rho-\rho_\hp}{}^2$}
  \psfrag{N}[cc][cc]{\scriptsize Number of elements $\#\TT$}
  \psfrag{y}[cc][cc]{}
  \psfrag{N^{-1}}[cr][cr]{\tiny $\#\TT^{-1}$}
  \psfrag{e=0 unif}[cr][cr]{\tiny$\eps=10^{n}$, $n=0$ unif}
  \psfrag{e=0 adaptive}[cr][cr]{\tiny$n=0$, adap}
  \psfrag{e=-4 unif}[cr][cr]{\tiny$n=-4$, unif}
  \psfrag{e=-4 adaptive}[cr][cr]{\tiny$n=-4$, adap}
  \psfrag{e=-06 unif}[cr][cr]{\tiny$n=-6$, unif}
  \psfrag{e=-06 adaptive}[cr][cr]{\tiny$n=-6$, adap}
  \psfrag{e=-08 unif}[cr][cr]{\tiny$n=-8$, unif}
  \psfrag{e=-08 adaptive}[cr][cr]{\tiny$n=-8$, adap}
  \psfrag{e=-12 unif}[cr][cr]{\tiny$n=-12$, unif}
  \psfrag{e=-12 adaptive}[cr][cr]{\tiny$n=-12$, adap}
  \psfrag{e=-16 unif}[cr][cr]{\tiny$n=-16$, unif}
  \psfrag{e=-16 adaptive}[cr][cr]{\tiny$n=-16$, adap}
  \psfrag{L2total/energy}[cc][cc]{\scriptsize
    $(\norm{u-u_\hp}{}^2 + \norm{\ssigma-\ssigma_\hp}{}^2 + \eps\norm{\rho-\rho_\hp}{}^2) / \norm{\uu-\uu_\hp}{E}^2$.}
  \psfrag{u/sigma}[cc][cc]{\scriptsize$\norm{u-u_\hp}{}^2 / \norm{\ssigma-\ssigma_\hp}{}^2$.}
  \psfrag{u/rho}[cc][cc]{\scriptsize$\norm{u-u_\hp}{}^2 / (\eps \norm{\rho-\rho_\hp}{}^2)$.}
  \centering
  \includegraphics[width=0.49\textwidth]{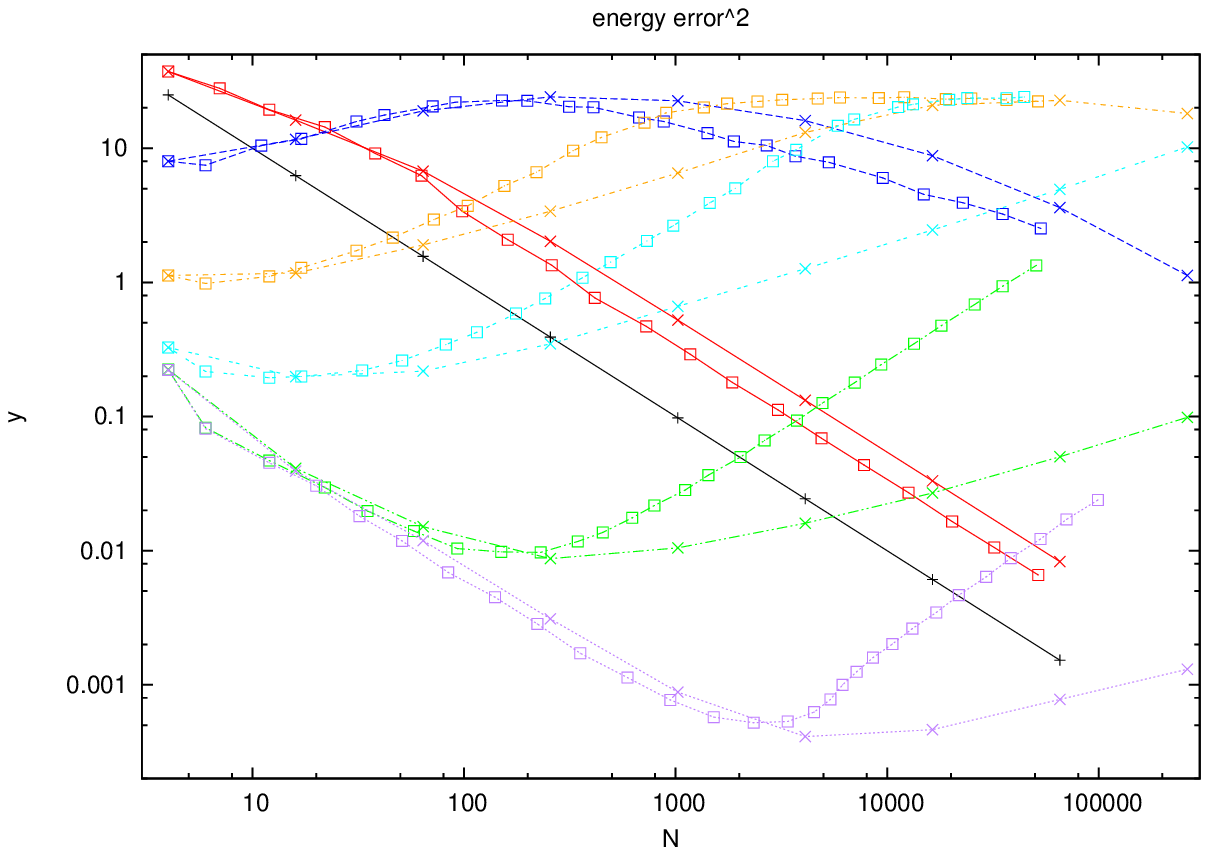}
  \includegraphics[width=0.49\textwidth]{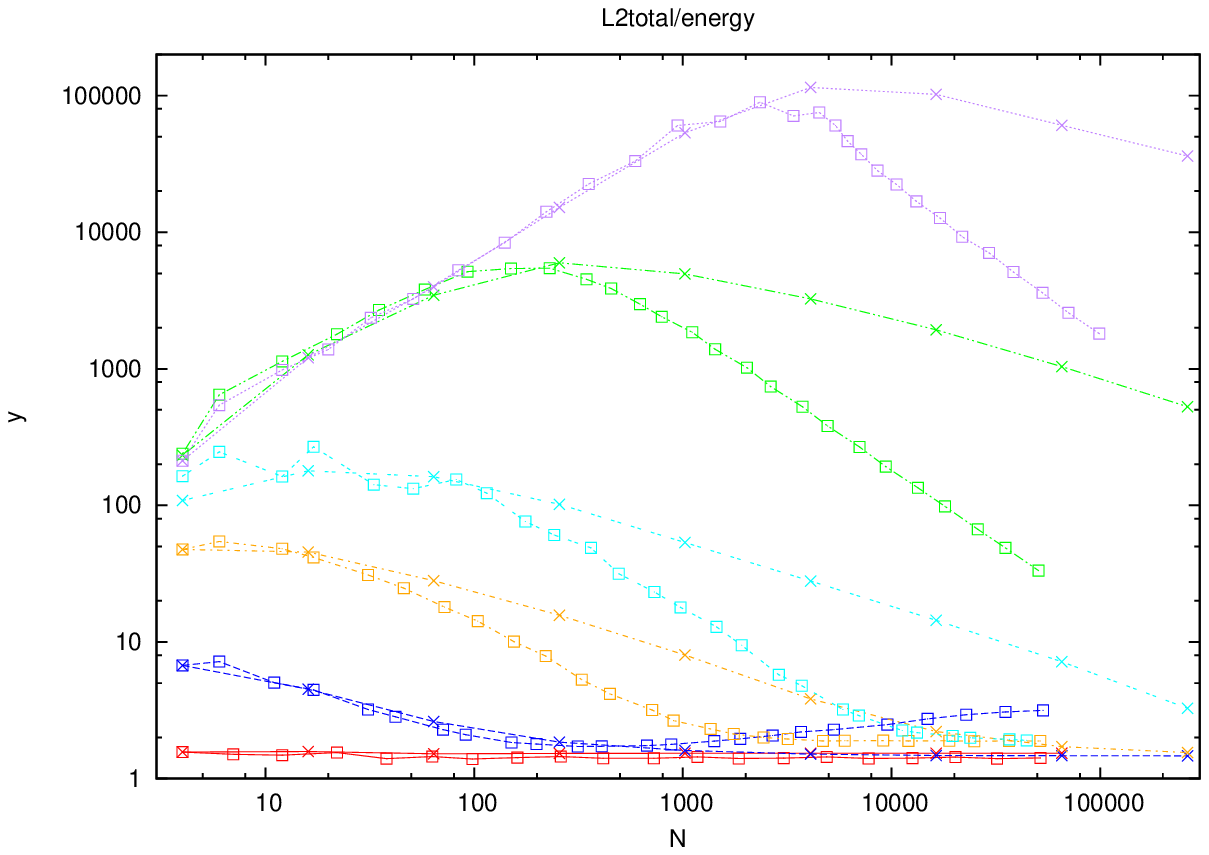}
  \includegraphics[width=0.49\textwidth]{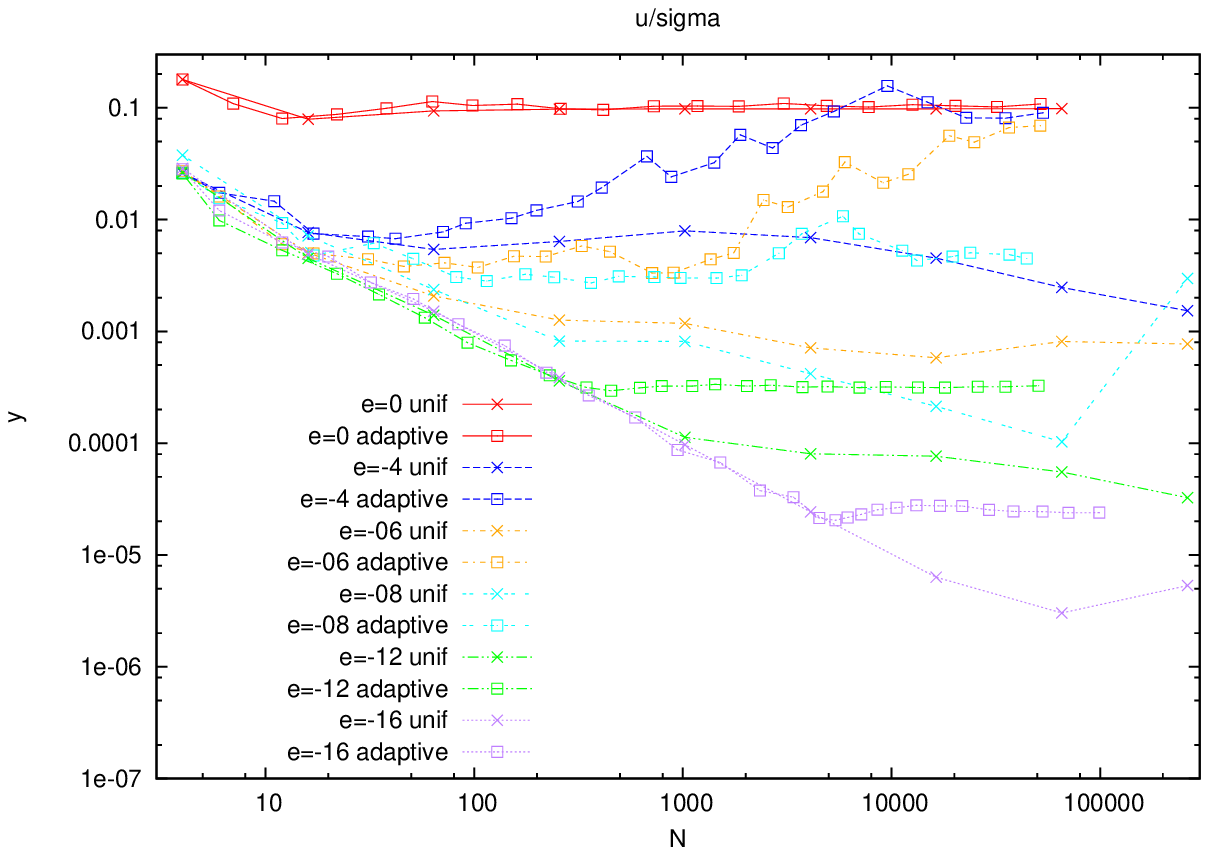}
  \includegraphics[width=0.49\textwidth]{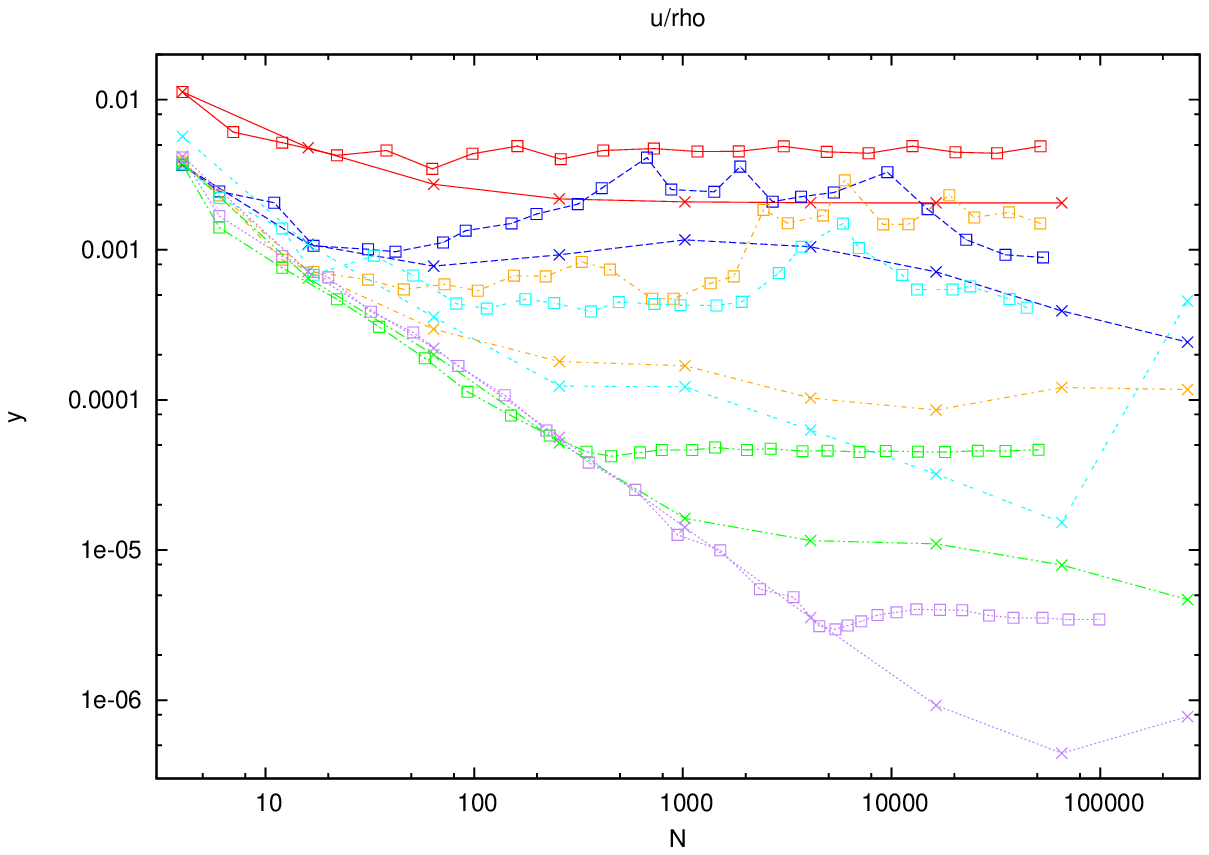}
  \caption{Manufactured solution from \S\ref{exp1}
           on uniform and adaptive meshes: energy error (upper left),
  error of field variables in balanced norm divided by energy error (upper right),
    $L_2$ error of $u$ divided by $L_2$ error of $\ssigma$ (lower left),
    $L_2$ error of $u$ divided $\eps^{1/2}$ times $L_2$ error of $\rho$ (lower right).
    The solid line in the upper left figure indicates $\OO(\#\TT^{-1})$.
    The parameter $\eps$ is chosen as $\eps=10^n$, where $n\in\left\{ 0,-4,-6,-8,-12,-16 \right\}$.
  }
  \label{fig:man}
\end{sidewaysfigure}
\begin{sidewaysfigure}[htb]
  \psfrag{L2 error of u^2}[cc][cc]{\scriptsize$\norm{u-u_\hp}{}^2$}
  \psfrag{L2 error of sigma^2}[cc][cc]{\scriptsize$\norm{\ssigma-\ssigma_\hp}{}^2$}
  \psfrag{L2 error of rho^2}[cc][cc]{\scriptsize$\norm{\rho-\rho_\hp}{}^2$}
  \psfrag{N^{-1}}[cr][cr]{\tiny $\#\TT^{-1}$}
  \psfrag{N}[cc][cc]{\scriptsize Number of elements $\#\TT$}
  \psfrag{y}[cc][cc]{}
  \psfrag{p=0 unif}[cc][cc]{\tiny $r=0$}
  \psfrag{p=1 unif}[cc][cc]{\tiny $r=1$}
  \psfrag{p=2 unif}[cc][cc]{\tiny $r=2$}
  \psfrag{p=3 unif}[cc][cc]{\tiny $r=3$}
  \psfrag{p=4 unif}[cc][cc]{\tiny $r=4$}
  \psfrag{p=5 unif}[cc][cc]{\tiny $r=5$}
  \psfrag{p=6 unif}[cc][cc]{\tiny $r=6$}
  \centering
  \includegraphics[width=0.49\textwidth]{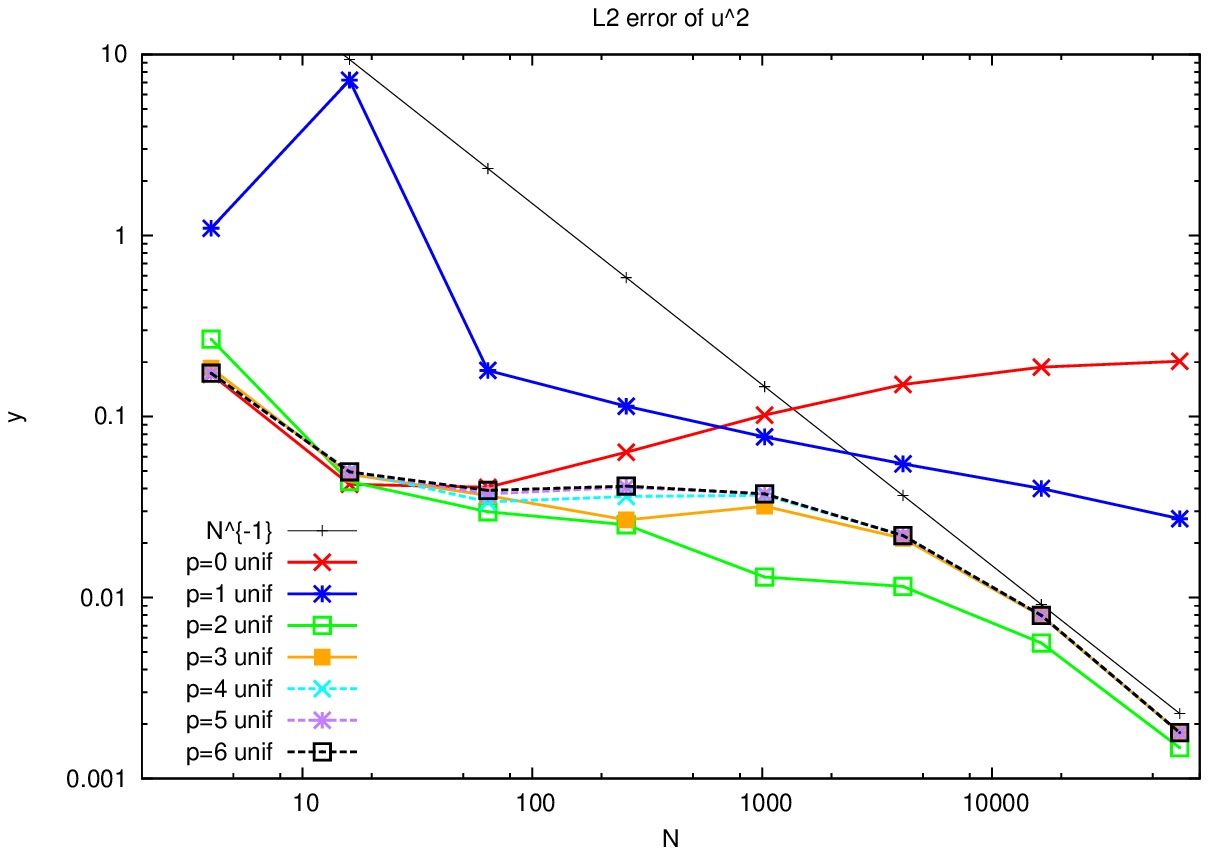}
  \includegraphics[width=0.49\textwidth]{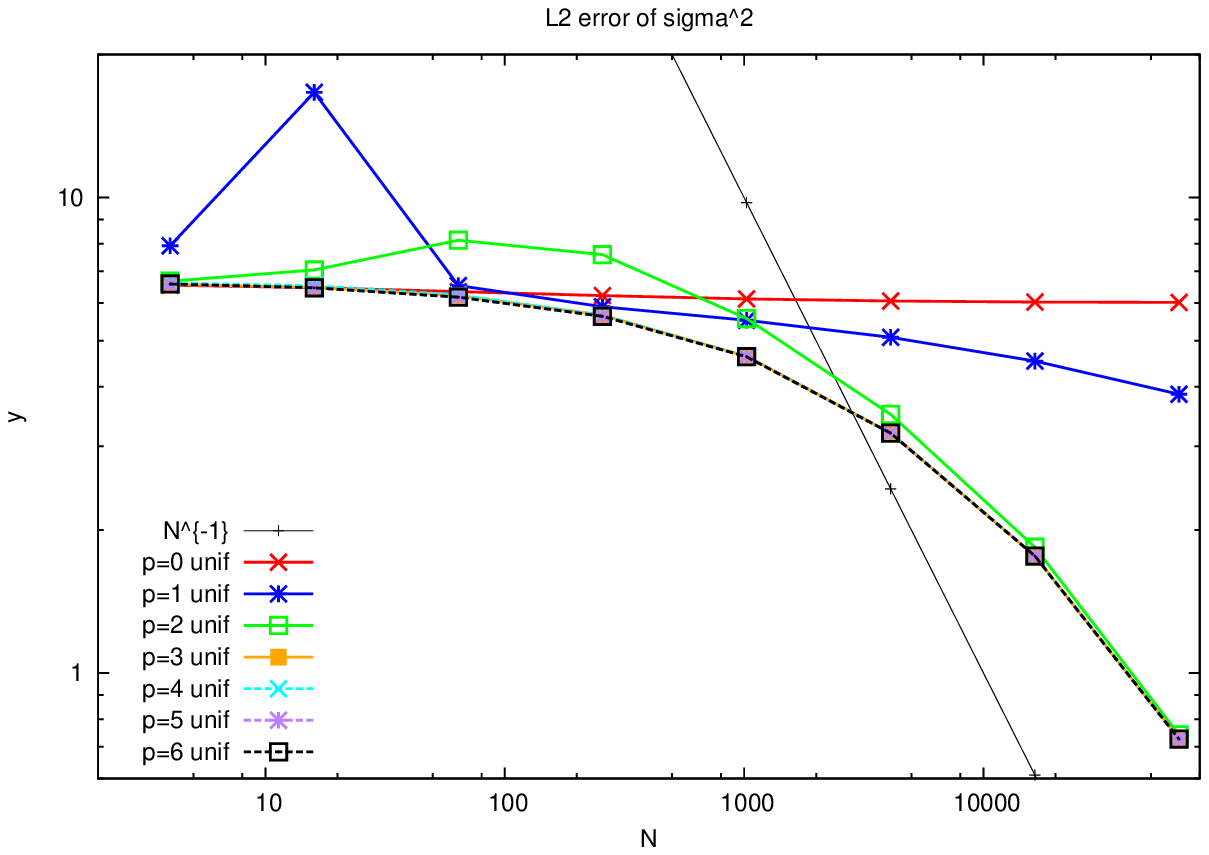}
  \includegraphics[width=0.49\textwidth]{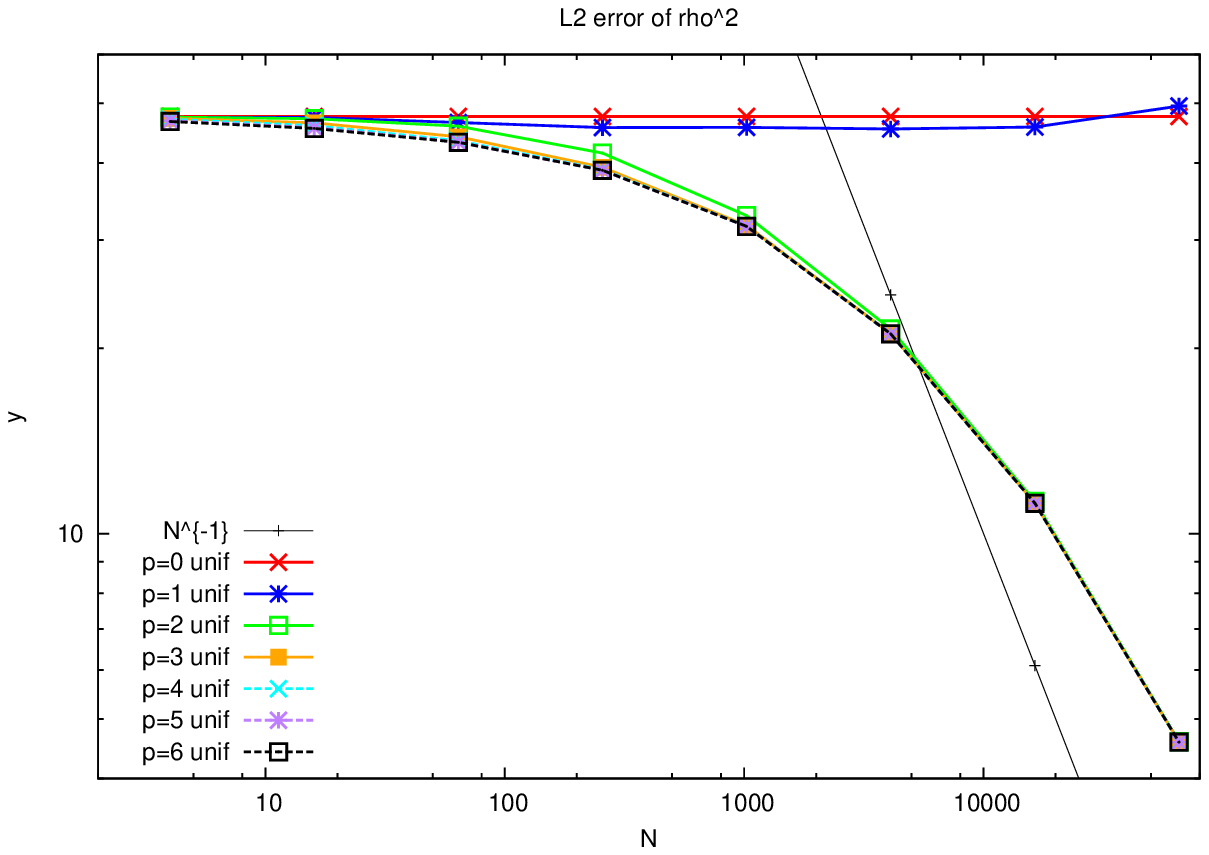}
  \caption{Manufactured solution from \S\ref{exp1}
           on uniform meshes with $\eps=10^{-4}$: The $L_2$ errors of the field variables
    $u$, $\ssigma$, and $\rho$ with different polynomial orders $r$ for the approximation of the test space.}
  \label{fig:man:testfunapx}
\end{sidewaysfigure}
\subsection{Problem with a layer not aligned to the mesh}\label{exp2}
We choose $\Omega = (0,1)^2$ and, with $\cc=(0.5,0.5)$ denoting the center of mass of $\Omega$,
\begin{align*}
  f(x,y) = \begin{cases} 1 & \text{ for } \abs{(x,y)-\cc}^2 < 0.1,\\
    0 & \text{ otherwise}.
  \end{cases}
\end{align*}
For small $\eps$, the solution $u$ is going to adjust to $f$ and hence we expect layers
inside $\Omega$ which cannot be aligned to the mesh. In addition, within the approximation properties
of our method, this problem is singular. More specifically, we have
$f\in H^{1/2-s}(\Omega)$ for all $s>0$, such that $\rho\in H^{1/2-s}(\Omega)$ for all $s>0$ only.
As we measure the error of $\rho$ in $L_2(\Omega)$, we expect a uniform convergence rate
of $\norm{\uu-\uu_\hp}{E}=\OO(\#\TT^{-1/4+s})$ for all $s>0$.
This is what we see in Figure~\ref{fig:unaligned:enrg} for uniform mesh refinement.
However, adaptive mesh refinement yields the optimal convergence rate $\OO(\#\TT^{-1/2})$. Note
that we plot squared quantities.
In Figure~\ref{fig:unaligned:mesh} we plot an adaptive mesh for $\eps=10^{-16}$
with approx. 12000 elements.
In Figure~\ref{fig:unaligned:sol}, we plot the $u$-component of solutions $\uu_\hp$
for different values of $\eps = 10^{ \{-16,-32,-64,-128\} }$ on adaptively refined meshes in
order to demonstrate the robustness of the approximations.
\begin{figure}[htb]
  \psfrag{energy error^2}[cc][cc]{\scriptsize$\norm{\uu-\uu_\hp}{E}^2$}
  \psfrag{L2error^2}[cc][cc]{\scriptsize$\norm{u-u_\hp}{}^2$}
  \psfrag{(e^1/4 H1-error)^2}[cc][cc]{\scriptsize$\norm{\ssigma-\ssigma_\hp}{}^2$}
  \psfrag{(e^3/4 H2-error)^2}[cc][cc]{\scriptsize$\eps\norm{\rho-\rho_\hp}{}^2$}
  \psfrag{N}[cc][cc]{\scriptsize Number of elements $\#\TT$}
  \psfrag{y}[cc][cc]{}
  \psfrag{N^{-1}}[cr][cr]{\tiny $\#\TT^{-1}$}
  \psfrag{N^{-1/2}}[cr][cr]{\tiny $\#\TT^{-1/2}$}
  \psfrag{e=0 unif}[cr][cr]{\tiny$n=0$, unif}
  \psfrag{e=0 adap}[cr][cr]{\tiny$n=0$, adap}
  \psfrag{e=-4 unif}[cr][cr]{\tiny$n=-4$, unif}
  \psfrag{e=-4 adap}[cr][cr]{\tiny$n=-4$, adap}
  \psfrag{e=-06 unif}[cr][cr]{\tiny$n=-6$, unif}
  \psfrag{e=-06 adap}[cr][cr]{\tiny$n=-6$, adap}
  \psfrag{e=-08 unif}[cr][cr]{\tiny$n=-8$, unif}
  \psfrag{e=-08 adap}[cr][cr]{\tiny$n=-8$, adap}
  \psfrag{e=-12 unif}[cr][cr]{\tiny$n=-12$, unif}
  \psfrag{e=-12 adap}[cr][cr]{\tiny$n=-12$, adap}
  \psfrag{e=-16 unif}[cr][cr]{\tiny$n=-16$, unif}
  \psfrag{e=-16 adap}[cr][cr]{\tiny$n=-16$, adap}
  \centering
  \includegraphics{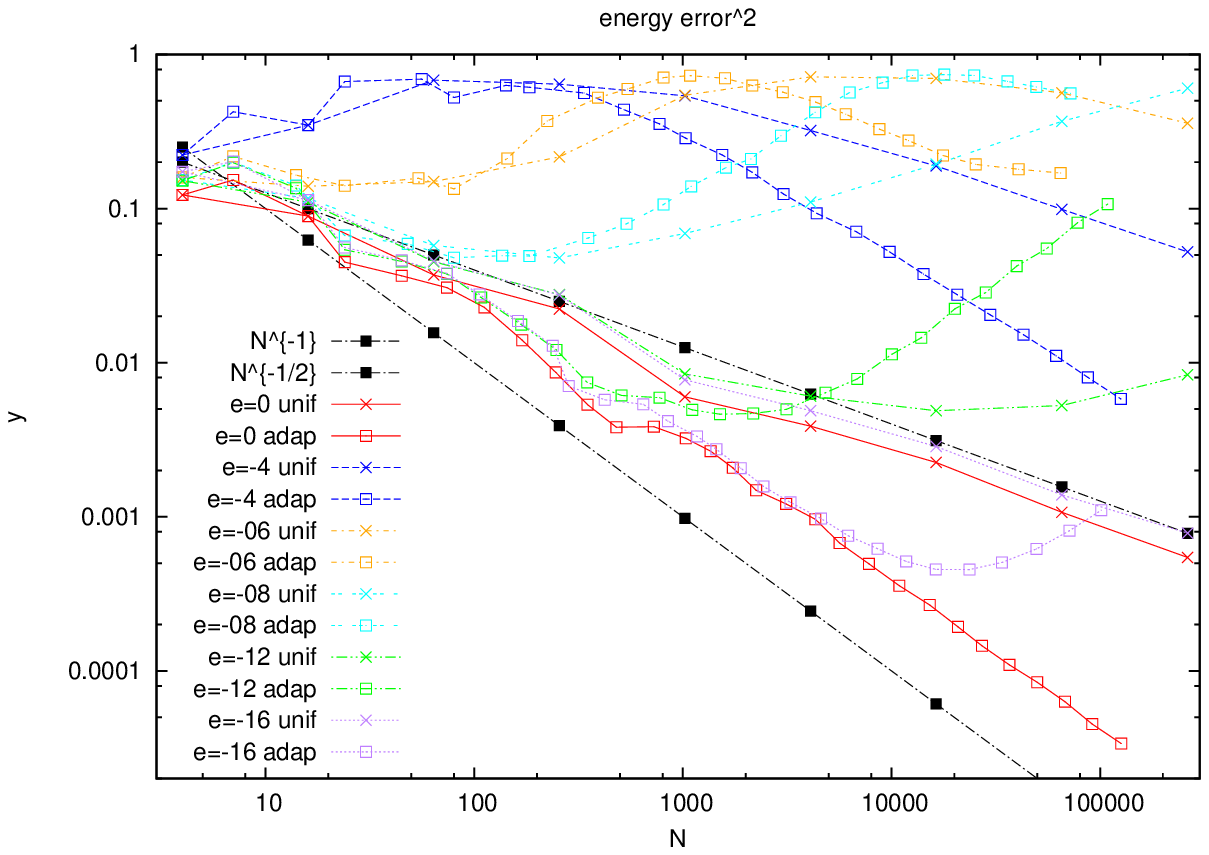}
  \caption{Energy error for solution from \S\ref{exp2}
           with unaligned layer on uniform and adaptive meshes.
  The parameter $\eps$ is chosen as $\eps=10^n$, where $n\in\left\{ 0,-4,-6,-8,-12,-16 \right\}$.
  }
  \label{fig:unaligned:enrg}
\end{figure}
\begin{figure}[htb]
  \centering
  \includegraphics[width=0.49\textwidth]{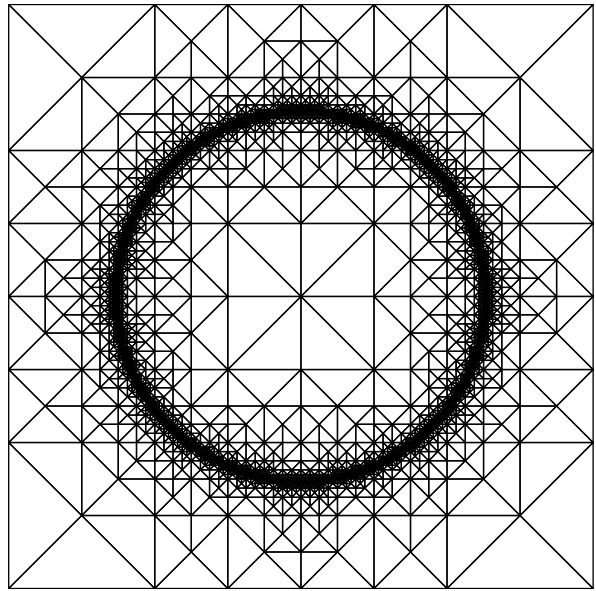}
  \caption{Adaptive mesh for problem from \S\ref{exp2}
           with approx. 12000 elements for $\eps=10^{-16}$.}
  \label{fig:unaligned:mesh}
\end{figure}
\begin{figure}[htb]
  \centering
  \includegraphics[width=0.49\textwidth]{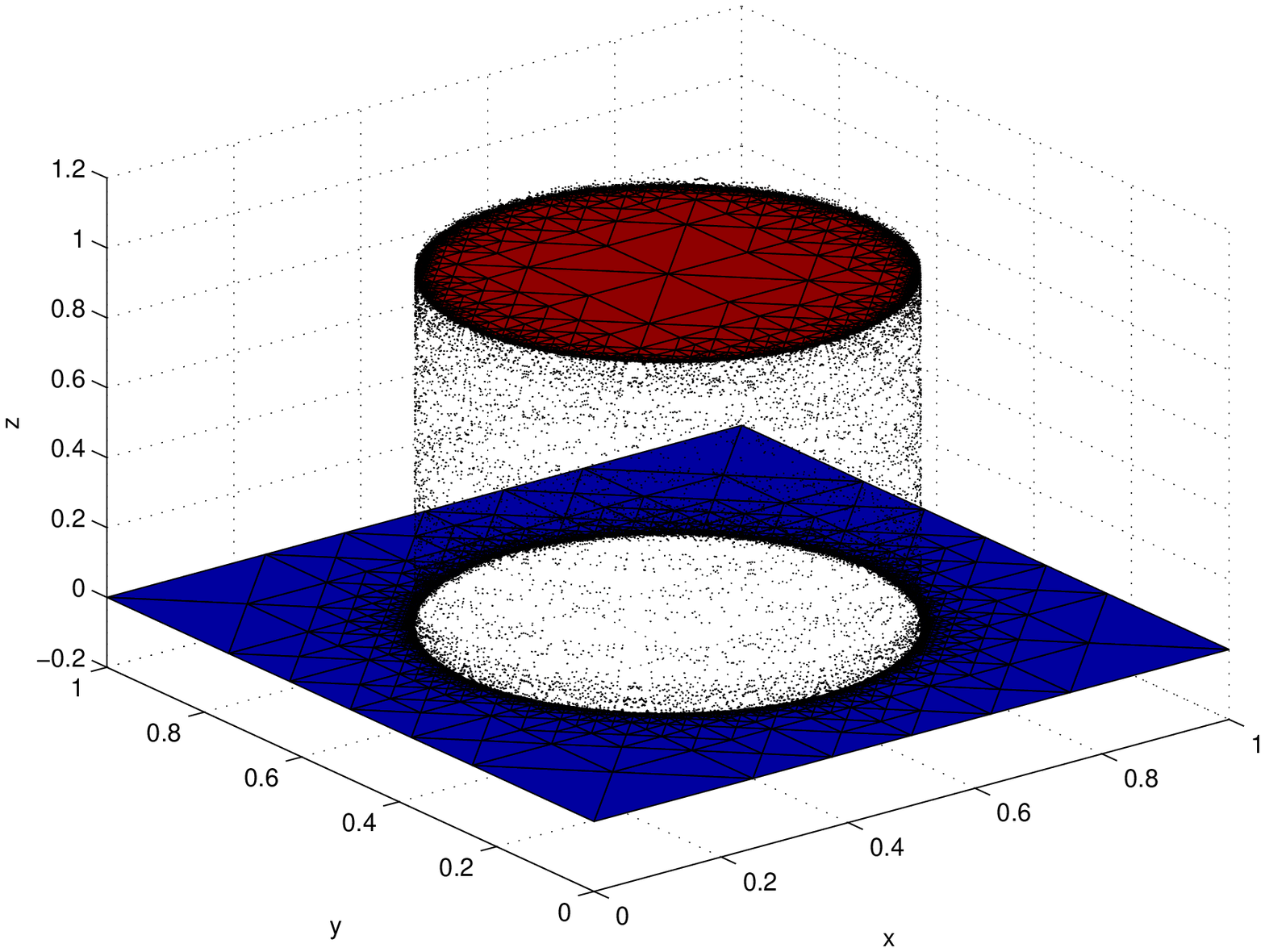}
  \includegraphics[width=0.49\textwidth]{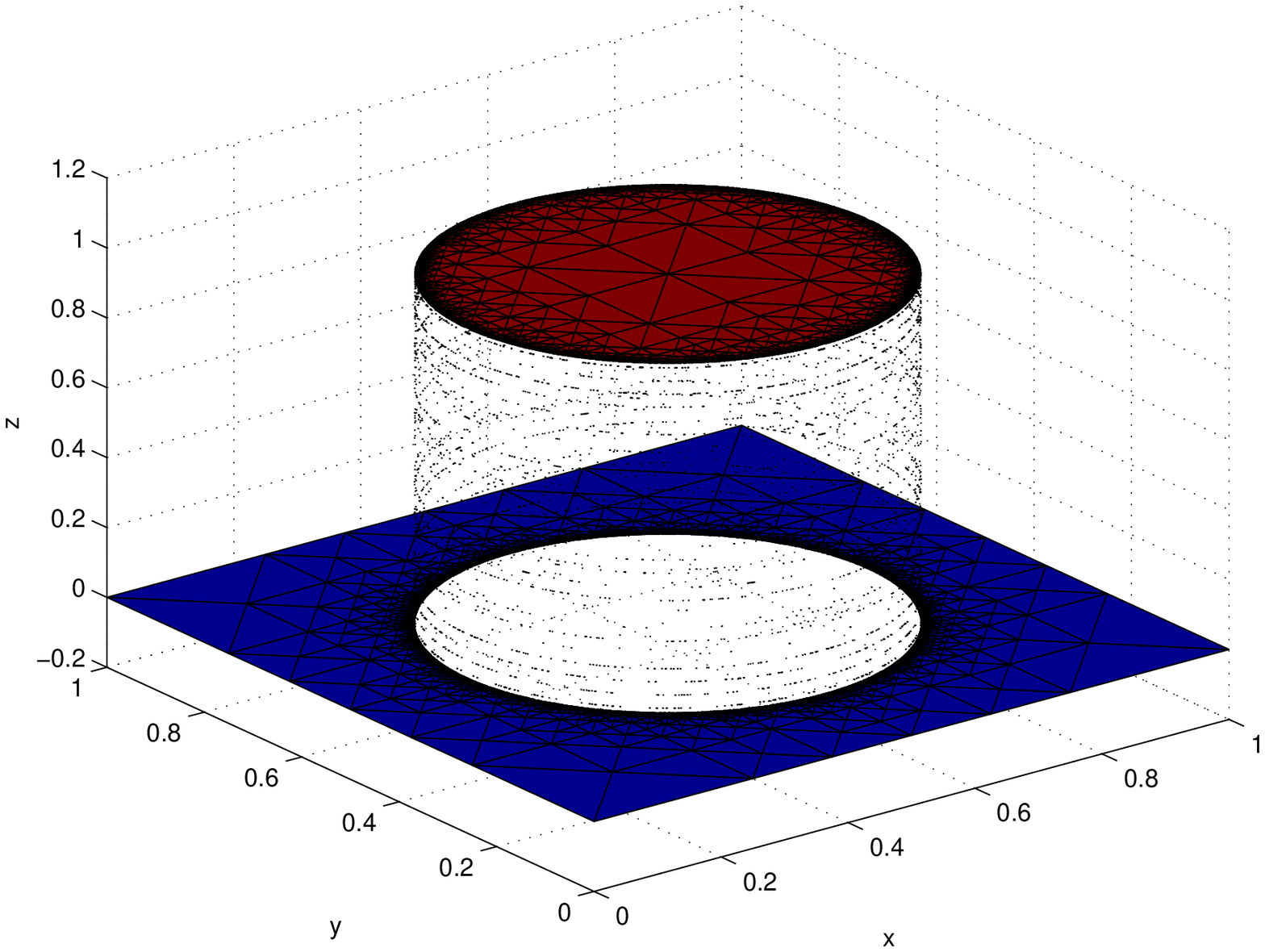}
  \includegraphics[width=0.49\textwidth]{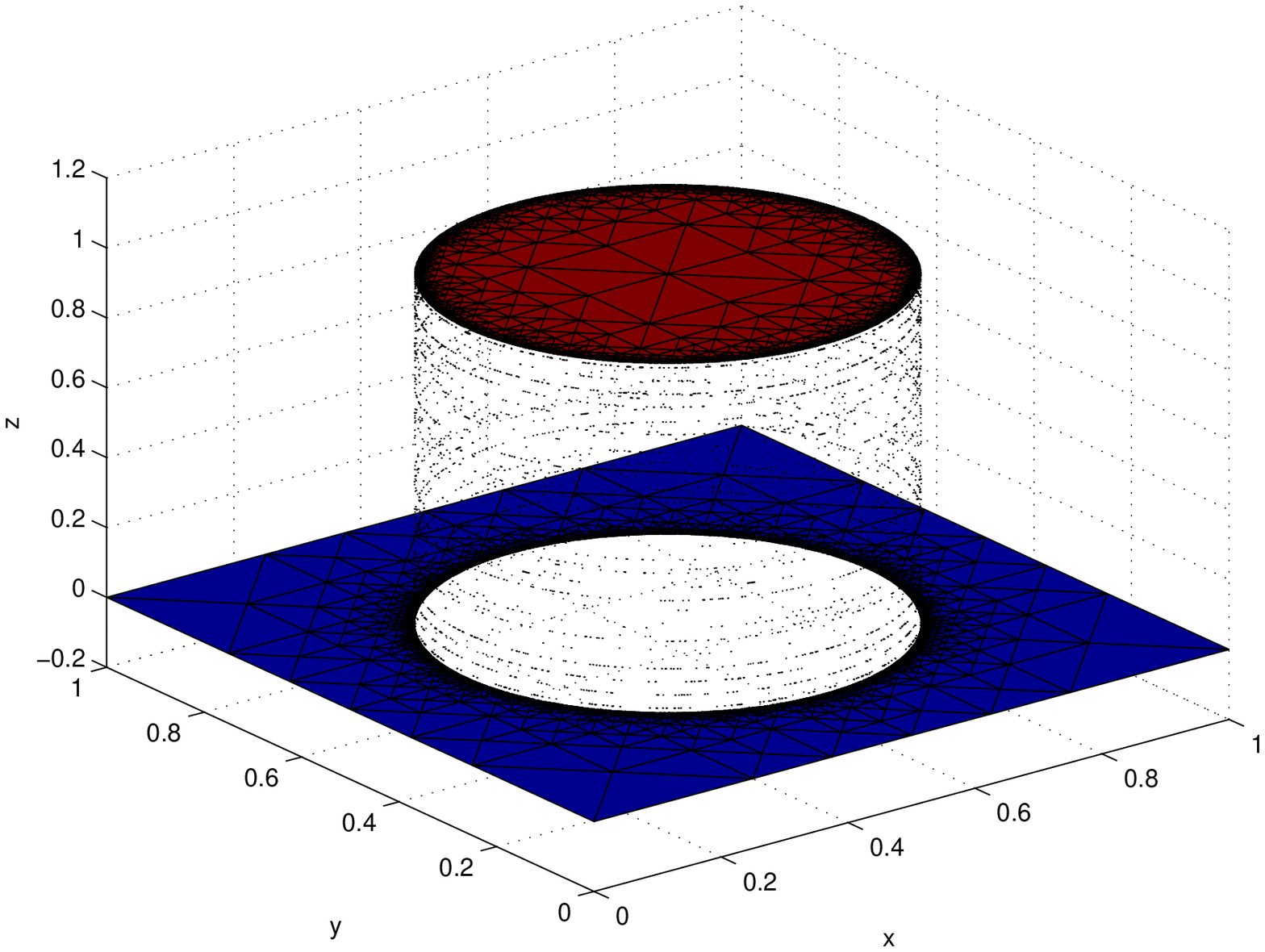}
  \includegraphics[width=0.49\textwidth]{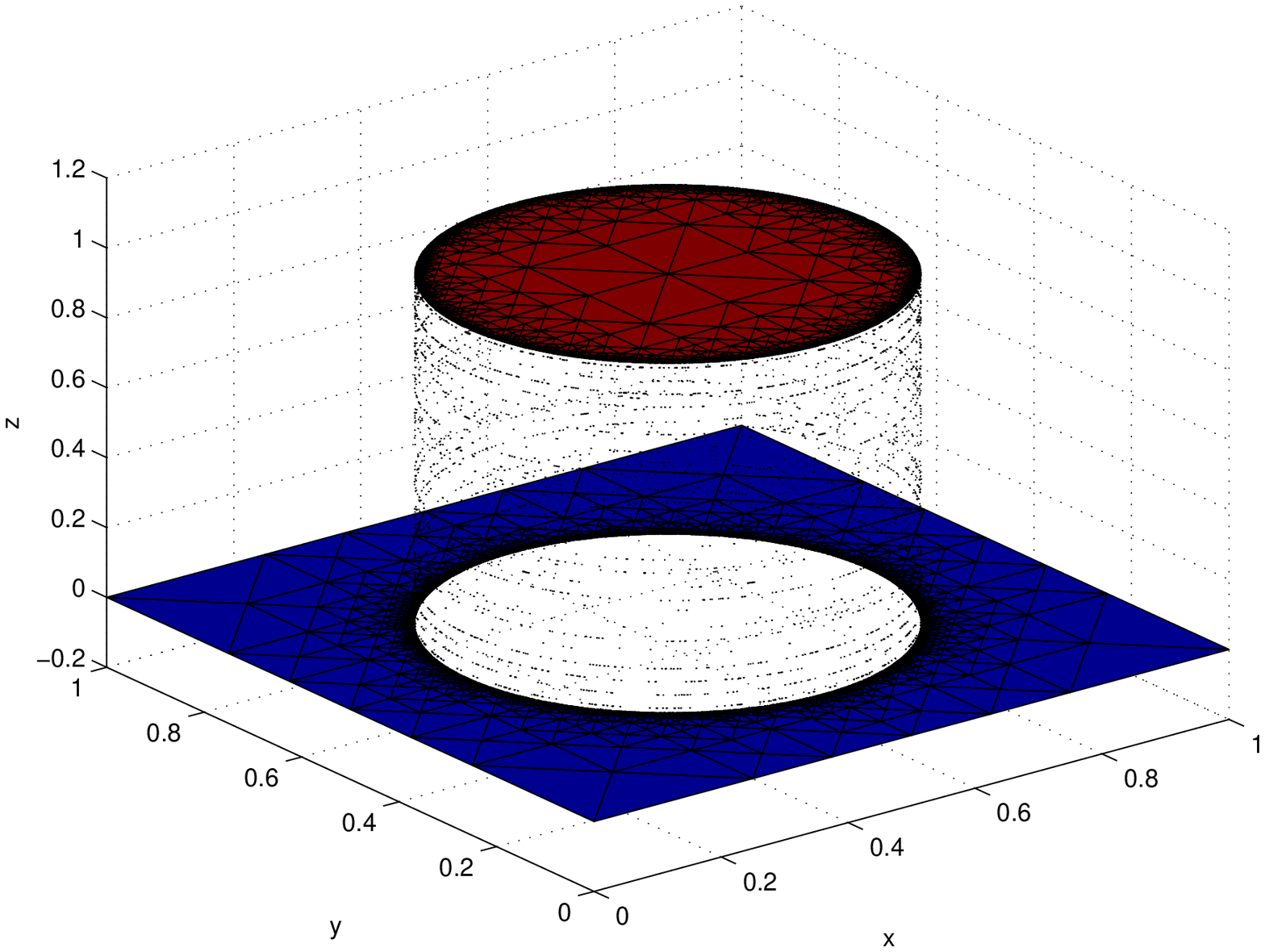}
  \caption{Approximation $u_\hp$ for problem from \S\ref{exp2}
           with unaligned layer on adaptive meshes with approx. 20000 elements
           for $\eps=10^{-16}$ (upper left),
           $\eps=10^{-32}$ (upper right),
           $\eps=10^{-64}$ (lower left),
           $\eps=10^{-128}$(lower right).}
  \label{fig:unaligned:sol}
\end{figure}
\subsection{Problem with a geometric singularity}\label{exp3}
We choose $\Omega$ to be an L-shaped domain and $f=1$. Therefore, we expect a geometric singularity at
the re-entrant corner which reduces the convergence rate. In Fig.~\ref{fig:singular:enrg}, we
plot the energy error $\norm{\uu-\uu_\hp}{E}$. We see that adaptive mesh refinement
regains the optimal convergence rate, as soon as the boundary layers are resolved.
In Fig.~\ref{fig:sing:mesh}, we plot adaptive meshes with approx. 10000 elements for
$\eps=1,10^{-4},10^{-8},10^{-16}$. While for $\eps=1$ we see the strong refinement
at the re-entrant corner, already for $\eps=10^{-4}$ the boundary layers dominate the mesh refinement
in this regime.
\begin{figure}[htb]
  \psfrag{energy error^2}[cc][cc]{\scriptsize$\norm{\uu-\uu_\hp}{E}^2$}
  \psfrag{N^{-1}}[cr][cr]{\tiny $\#\TT^{-1}$}
  \psfrag{N^{-1/2}}[cr][cr]{\tiny $\#\TT^{-1/2}$}
  \psfrag{e=0 unif}[cr][cr]{\tiny$n=0$, unif}
  \psfrag{e=0 adap}[cr][cr]{\tiny$n=0$, adap}
  \psfrag{e=-4 unif}[cr][cr]{\tiny$n=-4$, unif}
  \psfrag{e=-4 adap}[cr][cr]{\tiny$n=-4$, adap}
  \psfrag{e=-06 unif}[cr][cr]{\tiny$n=-6$, unif}
  \psfrag{e=-06 adap}[cr][cr]{\tiny$n=-6$, adap}
  \psfrag{e=-08 unif}[cr][cr]{\tiny$n=-8$, unif}
  \psfrag{e=-08 adap}[cr][cr]{\tiny$n=-8$, adap}
  \psfrag{e=-12 unif}[cr][cr]{\tiny$n=-12$, unif}
  \psfrag{e=-12 adap}[cr][cr]{\tiny$n=-12$, adap}
  \psfrag{e=-16 unif}[cr][cr]{\tiny$n=-16$, unif}
  \psfrag{e=-16 adap}[cr][cr]{\tiny$n=-16$, adap}
  \centering
  \includegraphics{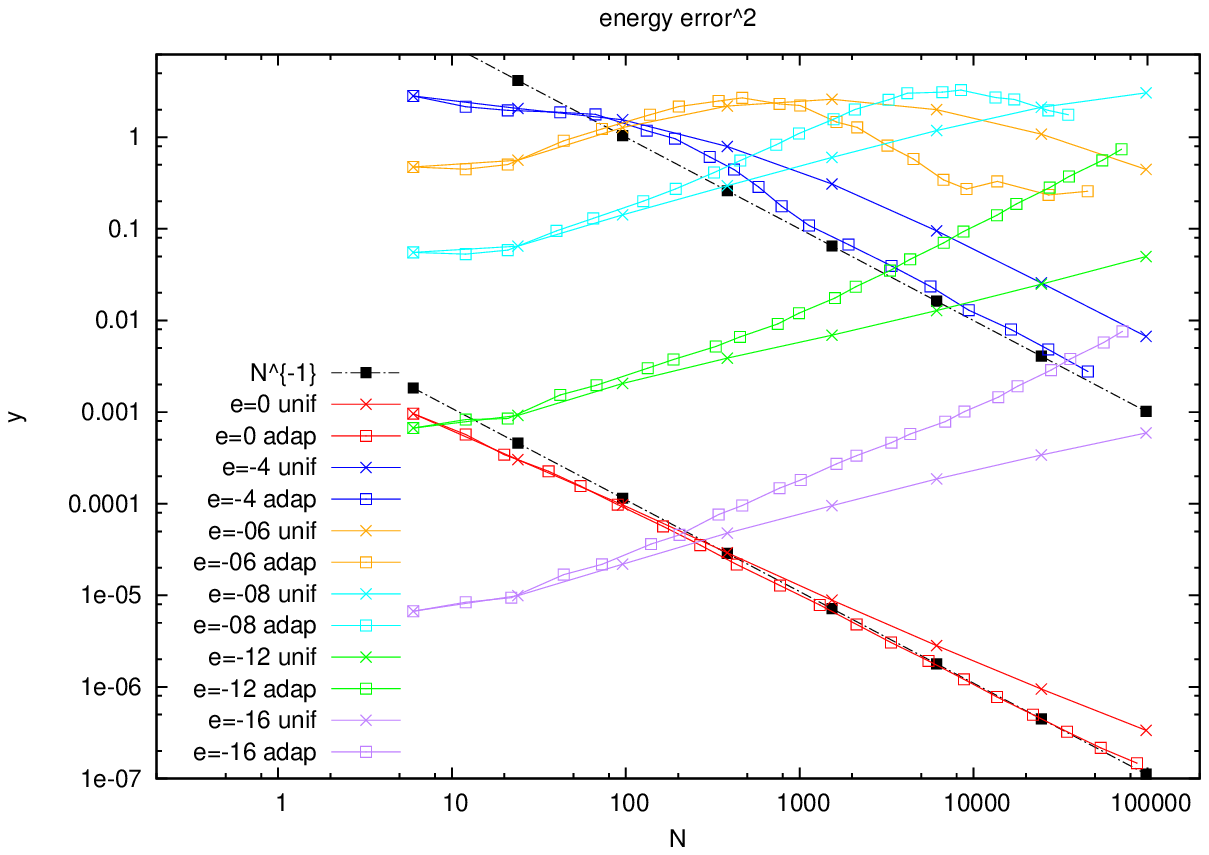}
  \caption{Energy error for singular solution from \S\ref{exp3} on uniform and adaptive meshes.
  The parameter $\eps$ is chosen as $\eps=10^n$, where $n\in\left\{ 0,-4,-6,-8,-12,-16 \right\}$.
  }
  \label{fig:singular:enrg}
\end{figure}
\begin{figure}[htb]
  \centering
  \includegraphics[width=0.49\textwidth]{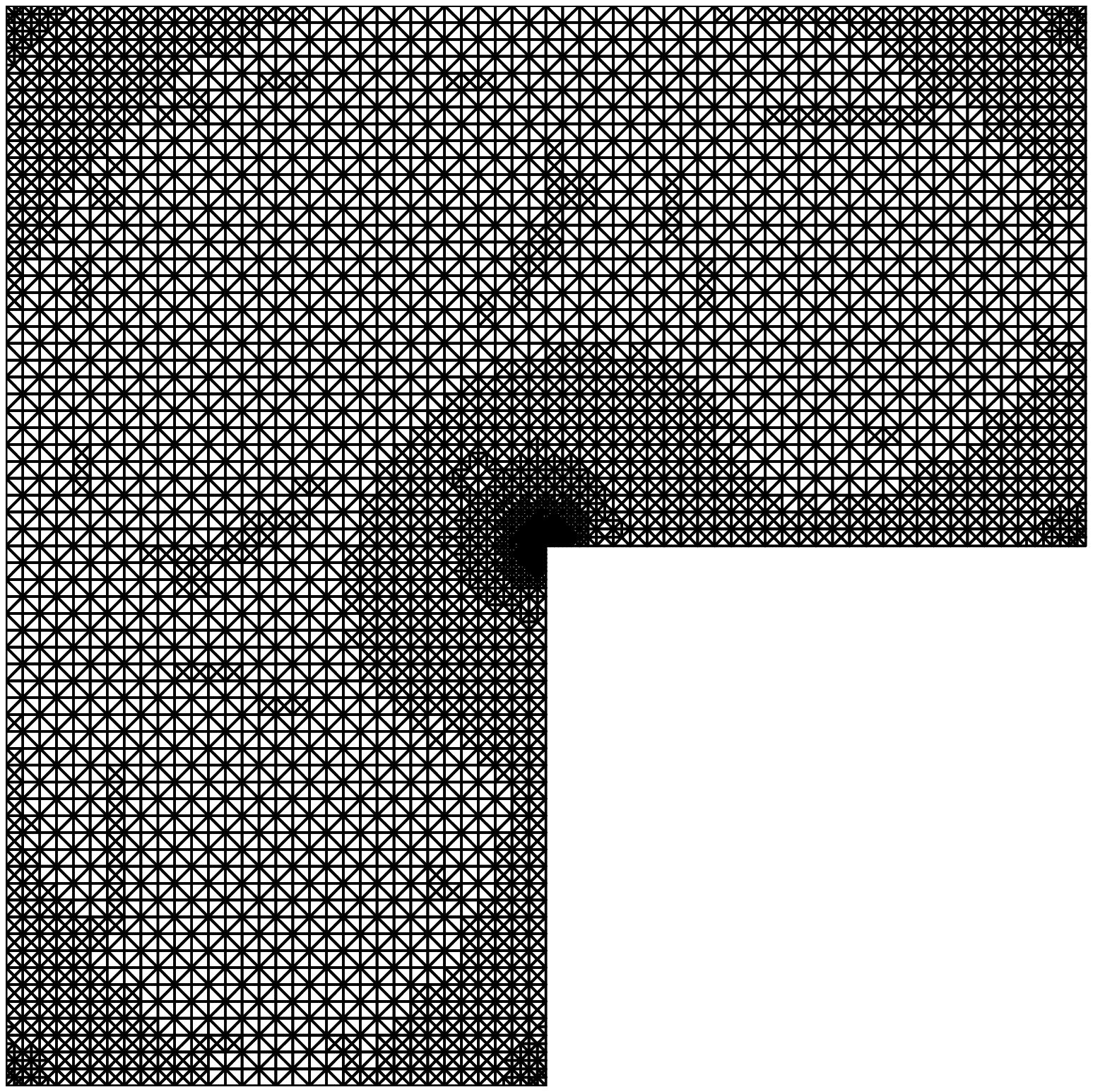}
  \includegraphics[width=0.49\textwidth]{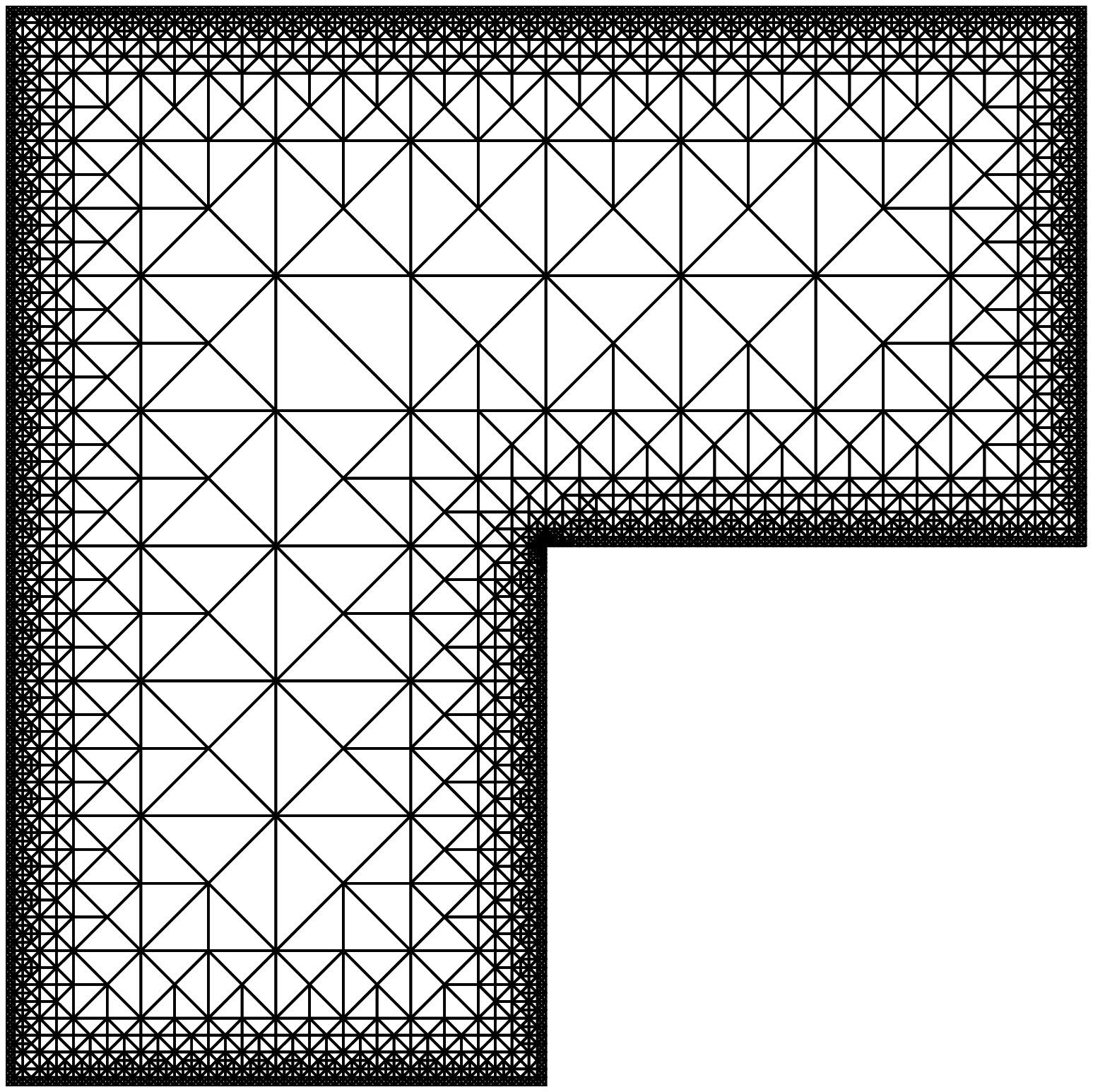}
  \includegraphics[width=0.49\textwidth]{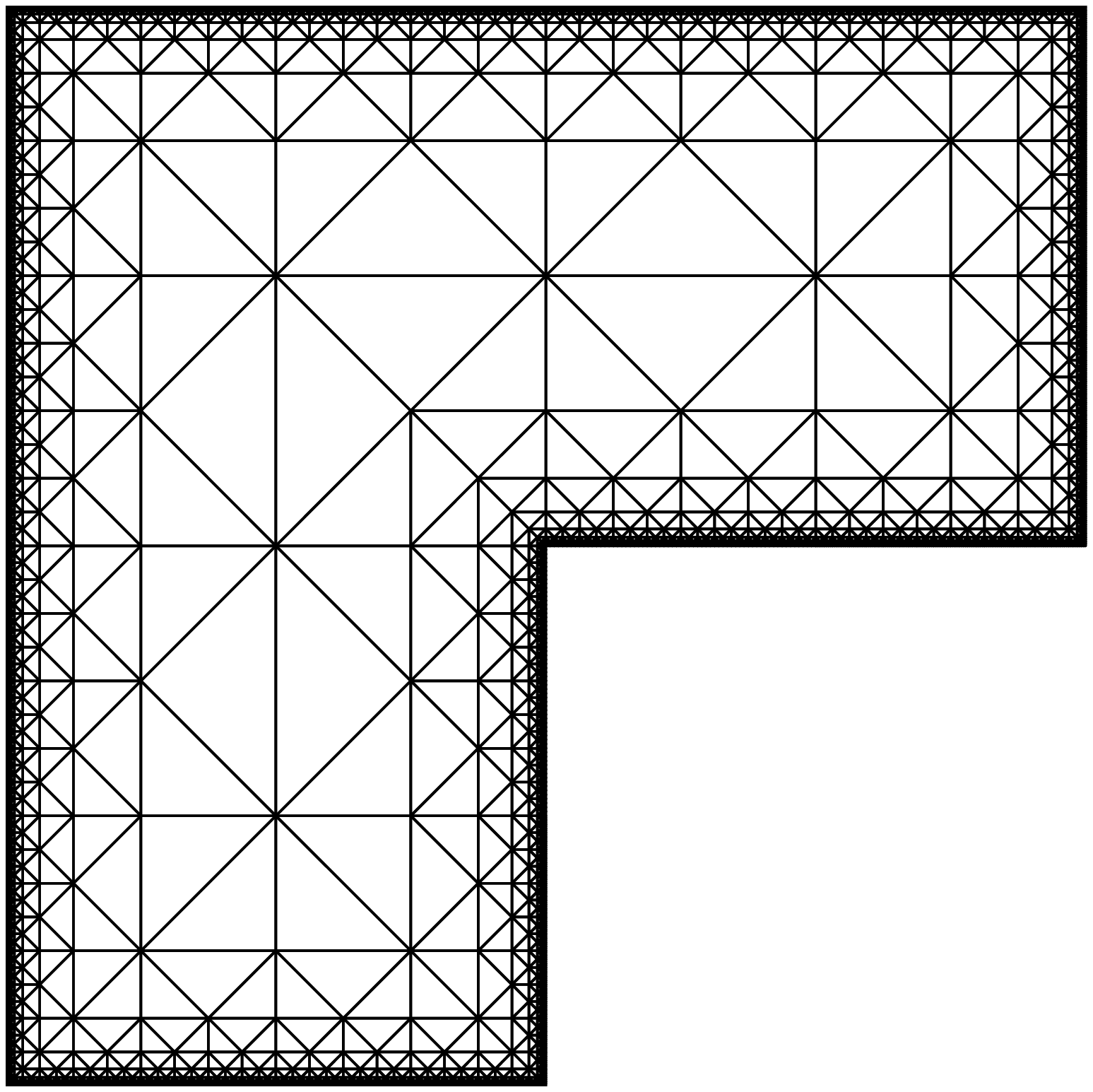}
  \includegraphics[width=0.49\textwidth]{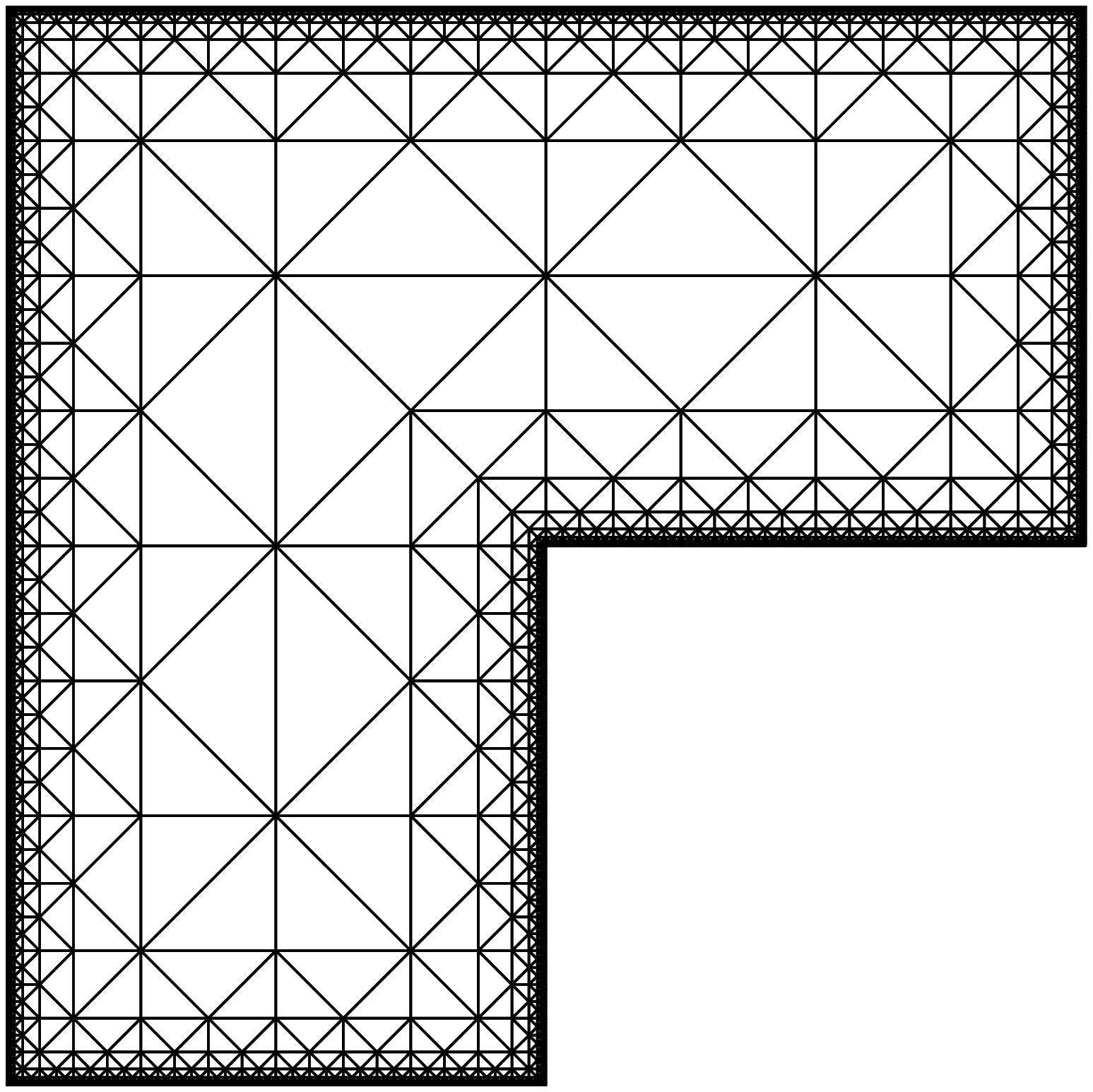}
  \caption{Adaptive meshes with approx. 10000 elements for singular solution from \S\ref{exp3}
    for $\eps=1,10^{-4},10^{-8},10^{-16}$.}
  \label{fig:sing:mesh}
\end{figure}


\bigskip
\noindent
{\bf Acknowledgment.} We thank Torsten Lin\ss\ for fruitful discussions on the subject of
singularly perturbed problems.

\bibliographystyle{siam}
\bibliography{bib}
\end{document}